\documentclass[11pt]{article}

\usepackage{amsmath}
\usepackage{amsfonts}
\usepackage{amssymb}
\usepackage{mathtools}
\usepackage{amsthm}
\usepackage{a4,a4wide}
\usepackage[dvipdfmx]{graphicx}
\usepackage{color}
\numberwithin{equation}{section}
\numberwithin{figure}{section}
\usepackage{hyperref}

\def\rightharpoonupfill@{\arrowfill@\relbar\relbar\rightharpoonup}

\newcommand{\norm}[1]{\left\Vert#1\right\Vert}

\newcommand{\Real}{\mathbb{R}}

\newcommand{\eps}{\epsilon}

\newcommand{\Ord}{\mathcal{O}}
\newcommand{\p}{\partial}

\newcommand{\Om}{\varOmega}
\newcommand{\La}{\lambda}

\newcommand{\U}{\mathcal{U}} 

\newcommand{\C}{\mathcal{C}}

\newcommand{\W}{\mathcal{W}}

\newcommand{\D}{\mathcal{D}}

\newcommand{\tg}{\!\!}
\newcommand{\Na}{\nabla}

\newcommand{\twoscale}{\xrightharpoonup{\rm 2-drift}}
\newcommand{\twoscales}{\xrightharpoonup{\rm 2s-drift}}

\newtheorem{remark}{Remark}[section]
\newtheorem{lemma}{Lemma}[section]

\newtheorem{theorem}{Theorem}[section]

\newtheorem{definition}[theorem]{Definition}

\usepackage{mathrsfs}
\DeclareMathAlphabet{\mathscrbf}{OMS}{mdugm}{b}{n}

\providecommand{\keywords}[1]{\textbf{\textit{Key words.}} #1}
\providecommand{\MSC}[1]{\textbf{\textit{MSC}} #1}

\title{Fast drift effects in the averaging of a  filtration combustion system -- a periodic homogenization approach }
\author{
Ekeoma Rowland Ijioma\thanks{MACSI, Department of Mathematics and Statistics, University of Limerick, Ireland.~(Email: e.r.ijioma@gmail.com)} 
~and~ Adrian Muntean\thanks{Department of Mathematics and Computer Science, Karlstad University, Sweden.~(Email: adrian.muntean@kau.se)} }

\begin{document}

\maketitle

\begin{abstract}
We target at the periodic homogenization of a semi-linear reaction-diffusion-convection system describing filtration combustion, where fast drifts are triggered by the competition between heat and mass transfer processes in an asymptotic regime of dominant convection. In addition, we consider the interplay between surface nonlinear chemical reactions and transport processes. To handle the oscillations occurring due to the heterogeneity of the medium,  we rely on the concept of two-scale convergence with drift to obtain, for suitably scaled model parameters, the upscaled system of equations together with effective transport parameters. The main difficulty is to treat the case of a coupled multi-physics problem. We proceed by extending the results reported by G. Allaire et al. and other related papers in this context to the case of coupled system of evolution equations pertinent to filtration combustion.
\end{abstract}

%\begin{keyword}
 
\smallskip
 \MSC{80A25, 35B27, 76M50, 80A32}

%\begin{AMS}
%\end{AMS}


%\end{keyword}

\section{Introduction}
Combustion in porous medium when air (or any other gaseous oxidizer) is injected has multiple applications in nowadays technology such as in self-propagating high-temperature synthesis \cite{Moore90}, smoldering combustion in normal and microgravity environments \cite{wang08, Olson98, Aldushin06}, etc. In this paper, we study the case of a dry porous medium, fully saturated by a gaseous mixture, and possibly allowing for both fast drift and fast combustion (fast gas-solid chemical reaction). An industrial application of a similar combustion regime can be found with combustions in fixed bed reactors in which a catalytic bed, preheated to a relatively high temperature, is exposed to a cold reacting mixture \cite{THULLIE95, Andrez99}. These studies exemplify a competition between mass and heat transport in porous media in which an exothermic reaction takes place.  The combined phenomena can be measured in the range of dominant Pecl\'et and Damk\"ohler numbers, which gives rise to dispersion in the proposed system. It is worth mentioning that the fast drift arises as a response of the system to a strong convection regime at the microscopic level whereas the fast combustion limit is a phenomenon common with non premixed flaming combustion; see, e.g., \cite{Faeth86}. 

Furthermore, the importance of dispersion phenomena in porous media combustion and other engineering applications is well-known in the literature; see, e.g., \cite{Taylor53,Salles93,Fatehi94,Choquet14}. In \cite{Pedras08}, thermal dispersion coefficients are calculated for an infinite porous medium by means of the volume averaging method; see also in this context \cite{HSU90,Sano11}. Also, the homogenization method has been applied in the derivation of dispersion coefficients; see, e.g., \cite{Moyne02,Auriault95}. However, in the framework of filtration combustion of solids, there is no detailed account of the dispersion effect on the transport parameters, which arises from strong competition between mass and thermal transport in the presence of chemical reaction. We stress that in such high temperature regimes, the use of measurement devices becomes impracticable, hence stimulations on the right effective transport coefficients need some theoretical insight.

Thus, we target the upscaling of combustion scenarios. We pay particular attention to capturing the effect of dominant P\'eclet and Damk\"ohler numbers, at the microscopic level, on the governing macroscopic (upscaled) combustion equations. Specifically, we focus on the structure of the effective thermal and mass dispersion tensors. The main mathematical difficulties arising in our context are fourfold:
\begin{itemize}
\item[$(i)$] the nonlinearity of the gas-solid chemical reaction kinetics of Arrhenius type;
\item[$(ii)$] the treatment of a coupled system of partial differential equations posed in high-contrasting microstructures;
\item[$(iii)$] the fast drift;
\item[$(iv)$] the coupled multi-physics and evolution system of equations.
\end{itemize}  
In our analysis, we will explore the nonlinearity of the problem (due to $(i)$) by mainly relying on its structural properties and by adopting the techniques by \cite{Harsha} to work in our setting. One possibility of avoiding the use of the nonlinearity in the Arrhenius law is to consider its linear approximation as suggested, for example, in \cite{Kagan}; nevertheless, the working hypotheses for such linear approximation are not so obvious. We borrow from our previous experience in handling situations like  $(ii)$  (see for instance \cite{Fatima12,Muntean12} for the treatment of a related scenario from chemical attack on concrete structures) and focus on the fast drift $(iii)$, which, at first sight, is an impediment to the classical theory of homogenization and on the new aspect $(iv)$. We adapt the working technique \cite{Harsha,Maro} to our combustion setting, so that we can deal with the aspect $(iii)$ in combination with $(iv)$; see also \cite{Mik,HH}. The method of two-scale convergence with drift as described in the latter reports is more appropriate for purely periodic velocity fields, with no dependence on the macroscopic variable. Therefore, a possible area of interest in the direction of problems with strong convection is the treatment of a more general flow field; specifically, in a locally-periodic setting following the recent work cited in \cite{Thomas17} and in the handling of the homogenization problem for multicomponent mass transport \cite{Allaire15}.

As possible extensions to our approach, we foresee the handling of locally periodic coverings of $\mathbb{R}^d$ ($d=2$ or $3$); considerations on the bounded domain case being however out of reach for the moment. Due to finite size effects, localizations of both heat and concentration seem unavoidable, at least for naive scalings, see e.g. \cite{JFA} for the evidence of localization in a scalar case.

In a forthcoming publication, we will also look into the case when liquid islands (which can eventually be perceived as randomly distributed defects) are initially present in the porous medium, which typically occurs in coal gasification or in-situ combustion in oil recovery, see e.g. \cite{HansBruining}.

The paper is organized as follows: we describe the geometry of the porous medium in Section \ref{filtration-model}, while Section \ref{Scaling} contains our scaling arguments. The model equations are listed in Section \ref{micro-pb}. The main result of this paper is the set of upscaled equations and effective coefficients (for the heat capacity, transport and chemical reaction) summarized in Theorem \ref{MR} and reported with details in Section \ref{main-result}. The proof of Theorem \ref{MR} is shown in Section \ref{proofMR}. 
%%%%%%%%%%%%%%%%%%%%%%%%%%%%%%%%%%%%%%
\section{Problem description}\label{filtration-model}
%%%%%%%%%%%%%%%%%%%%%%%%%%%%%%%%%%%%%%
Basic modeling considerations in combustion can be looked up, for instance, in \cite{buck85,Mimura09} and references cited therein. The pore scale description of the filtration model studied in this framework, has been previously introduced in \cite{Ijioma13,Ijioma15b}. We recall here only the main ingredients of the model. In filtration combustion processes, there are predominantly two competing transport processes, the transport of heat and the transport of mass of a gaseous mixture, which are governed by convection-diffusion equations coupled with appropriate chemical reactions to exhibit the correct phenomena at the macroscopic scale. The competition allows for a thermal-diffusive transport through the porous medium, say $\Om$, which we consider to be infinite and representative of the physical material of interest $\Om$ has essentially two distinct parts--a skeletal structure made of an array of periodically placed reactive solid obstacles that is complemented with pores. In the context of this study, we assume this medium is saturated with a gaseous mixture. The gas velocity, denoted by $b$, is taken to be purely periodic and given. We denote the region of the porous medium occupied by the gaseous mixture by $\Om_{\rm g}$ while the remainder, representing the region of the solid obstacles, is denoted by $\Om_{\rm s}$; a gas-solid interface separating the two regions is denoted by $\p\Om_{\rm s}$. Convective transport is modeled through the velocity of the gas in which an oxidizer (component of the gaseous mixture that participates in a chemical reaction) of mass concentration $\C(t,x)$, and heat with temperatures, $T_{\rm g}(t,x)$ and $T_{\rm s}(t,x)$ respectively, distributed in the gas and solid regions, are conveyed. Here, $x\in \Om$ denotes the macroscopic space variable and $t$, the time variable. The diffusive and conductive transports in the two regions are determined by the molecular diffusion $D$, for the mass transport of the oxidizer and by the thermal conductivities, $\lambda_{\rm g}$ and $\lambda_{\rm s}$, for the transports of heat in the gas and solid regions, respectively. The interplay between the oxidizer and the reactive solid obstacles in the presence of heat initiates a chemical reaction. The combustion mechanism is governed by an Arrhenius type kinetics at the gas-solid interfaces. To keep the presentation simple, we assume that heat radiation and adsorption of the gaseous mixture on the solid surface are negligible. 
Let the temperature in the porous medium $\Om$, be denoted by $T(t,x).$ We decompose $T(t,x)$ in the two regions as follows:
\begin{align}\label{extT}
T(t,x)= 
\begin{dcases}
T_{\rm s}(t,x),& x \in \Om_{\rm s},\\
T_{\rm g}(t,x),& x\in \Om_{\rm g}.
\end{dcases}
\end{align} 
Depending on the situation, it is sometimes convenient to work with the temperature $T(t,x)$, while at some other times the use of {\em two temperatures}\footnote{The decomposition of the temperature should not be confused with that one typically arising in the theory of heat conduction involving two temperatures (i.e. the {\em conductive} and the {\em thermodynamic} temperatures). In our scenario, the conductive and thermodynamic temperatures coincide; compare \cite{Chen}. } is more convenient. 
%%%%%%%%%%%%%%%%%%%%%%%%%%%%%%%%%%%%%%%%%
\subsection{Mathematical model}
%%%%%%%%%%%%%%%%%%%%%%%%%%%%%%%%%%%%%%%%%
Denote by $t_f\in (0,\infty)$ the final observation time of the combustion process. The balance of heat and mass transport in the porous medium is described by two convection-diffusion-like equations posed in $\Om_{\rm g}$ and a heat conduction equation in $\Om_{\rm s}$, viz.
\begin{align}\label{model1}
\begin{dcases}
c_{\rm g}\dfrac{\p T_{\rm g}}{\p t} +c_{\rm g}  b\!\cdot\!\Na T_{\rm g} - \Na\!\cdot\!(\La_{\rm g}\Na T_{\rm g})=0,&\mbox{in $(0,t_f)\times \Om_{\rm g},$}\\
c_{\rm s}\dfrac{\p T_{\rm s}}{\p t} - \Na\!\cdot\!(\La_{\rm s}\Na T_{\rm s})=0,&\mbox{in $(0,t_f)\times \Om_{\rm s},$}\\
\dfrac{\p \C}{\p t} +  b\!\cdot\!\Na \C - \Na\!\cdot\!(D\Na \C)=0,&\mbox{in $(0,t_f)\times \Om_{\rm g}$}.
\end{dcases}
\end{align}
The system \eqref{model1} is coupled at the gas-solid interface via the following flux balances:
\begin{eqnarray}\label{model2}
\La_{\rm g}\Na T\!\cdot  n = \La_{\rm s}\Na T\!\cdot  n + QW(T,\C), &\mbox{ on $(0,t_f)\times \p\Om_{\rm s}$},\\
D\Na \C\!\cdot  n =-W(T,\C), &\mbox{ on $(0,t_f)\times \p\Om_{\rm s}$},
\end{eqnarray}
where $Q>0$ represents the heat release and $W:\Real\times\Real \rightarrow \Real$, the surface production term, is defined as 
\begin{align}\label{ratelaw}
W(\alpha,\beta) = A\beta f(\alpha),~\mbox{with $f(\alpha):=\exp\Big(-\dfrac{T_{\textrm a}}{\alpha}\Big).$}
\end{align}
\eqref{ratelaw} is a first-order Arrhenius kinetics with a pre-exponential factor $A>0$ and an activation temperature $T_{\textrm a}>0$.~The form of \eqref{ratelaw} agrees with a constant solid fuel assumption. We assume continuity of temperature across the interface, ie.,
\begin{eqnarray}\label{model3}
T_{\rm g} = T_{\rm s}\quad\mbox{on $(0,t_f)\times \p\Om_{\rm s}$}.
\end{eqnarray}
%%%%%%%%%%%%%%%%%%%%%%%%%%%%%%%%%%%%%
\subsection{Description of the porous medium}
%%%%%%%%%%%%%%%%%%%%%%%%%%%%%%%%%%%%%
In this section, we describe the structure of the heterogeneous porous medium of interest. 
\begin{figure}[!h]
\centering
\includegraphics[scale=0.3]{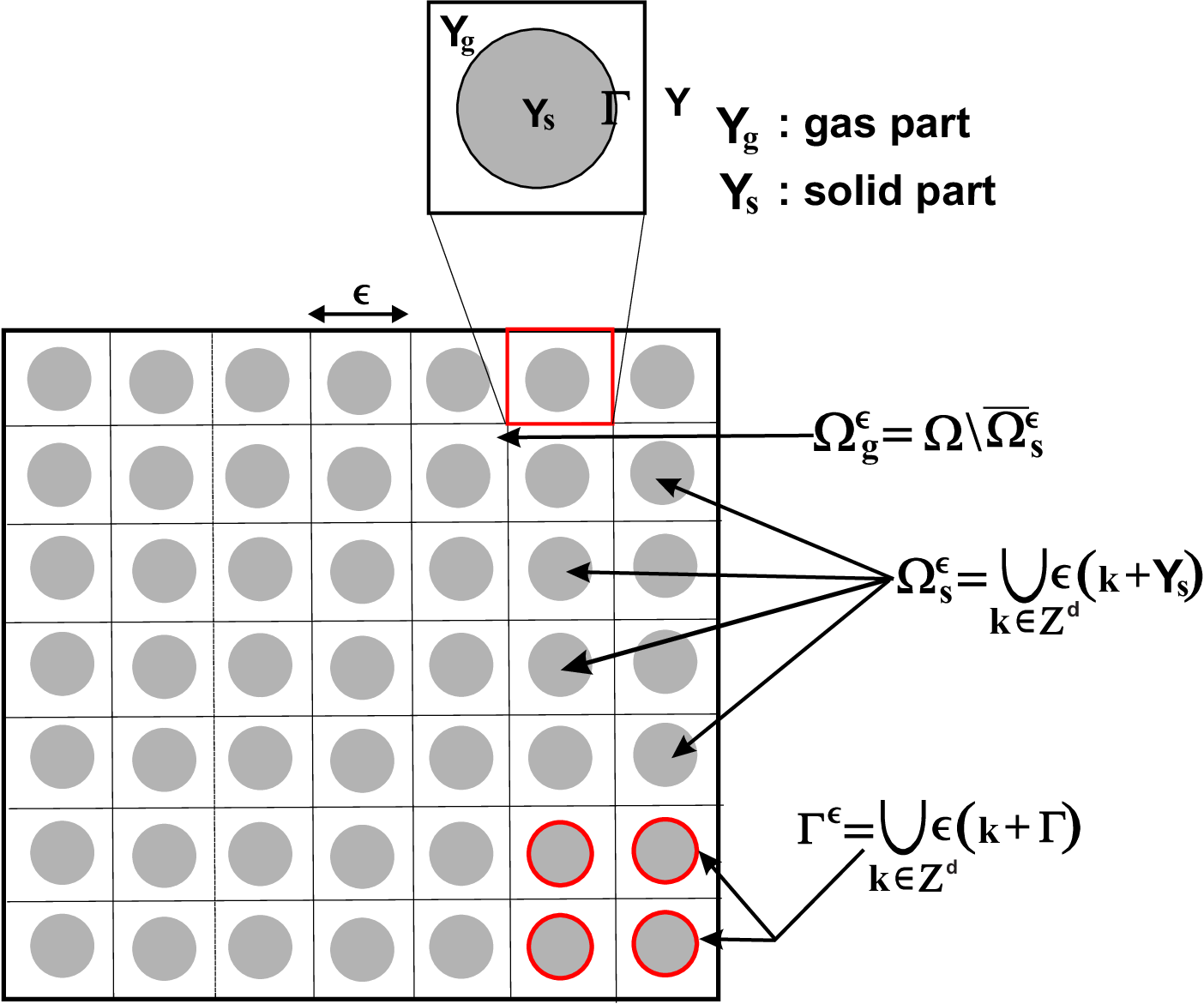}
\caption{Two-dimensional periodic domain setting the stage for the homogenization and the reference cell. The solid obstacles are represented by the gray disks with the boundaries in red. The remaining part of the medium consists of the gas domain.}
\label{Figure2}
\end{figure}
%%%%%%%%%%%%%%%%%%%%%%%%%%%%%%%%%%%%%%%%%%%%%%%%%%%%%%%%%%%%%
% GEOMETRY AND NOTATIONS
%%%%%%%%%%%%%%%%%%%%%%%%%%%%%%%%%%%%%%%%%%%%%%%%%%%%%%%%%%%%%
We assume $\Om^\eps$ to be an $\eps$-periodic unbounded open set of $\Real^d$, where $d\in \{2, 3\},$ is the space dimension. $\eps$ is a scale parameter defined as the ratio between the period in the arrangement of the reactive obstacles and a characteristic macroscopic length scale of interest. The porous medium is subdivided into a periodic distribution of cells $Y^\eps$, with each cell being equal and defined up to a scaled translation of a reference unit cell $Y=[0,1]^d$. The unit cell consists of two distinct parts: a gas-filled part and  a solid part. We denote by $Y_{\rm s}$ the solid part of $Y$. Denote by $\Gamma$ its smooth boundary. The gas-filled part is a smooth connected open set defined by $Y_{\rm g}=Y\setminus \overline{Y}_{\rm s}.$ We represent the collection of the periodically distributed solid parts by $Y^\eps_{\textrm s,k}=\eps \big(k+Y_{\rm s}\big),$ and thus the periodically distributed boundaries becomes $\Gamma^\eps_{k}=\eps \big(k+\Gamma\big)$ for each multi-index $k\in \mathbb{Z}^d$. By these notations, the ensemble of the periodically translated solid parts, the matrix of connected gas-filled parts and the ensemble of the solid boundaries can be written as
\begin{eqnarray}
\Om^\eps_{\rm s}=\bigcup_{k\in\mathbb{Z}^d}Y^\eps_{\textrm s,k},~~\Om^{\eps}_{\rm g}=\Om^\eps\setminus\overline{\Om^{\eps}_{\rm s}},~~\Gamma^{\eps}=\bigcup_{k\in\mathbb{Z}^d}\Gamma^{\eps}_{k},
\end{eqnarray}
Compatible with the $\eps$-periodic representation of the porous medium, we assume that all involved physical quantities and functions are rapidly oscillating. That is, for a given function $\psi$, we write $\psi^{\eps}(x) = \psi(x/\eps)$ for all macroscopic variable $x\in\Om^\eps$ with $y=x/\eps$ representing the microscopic variable. We also assume that the function $\psi(y)$ and all physical quantities can be extended by $Y$-periodicity to the whole of $\mathbb{R}^d$ (with a period $\eps$).~Thus, the triplet ($T, \C, b$) can be written as ($T^{\eps}, \C^{\eps}, b^{\eps}$). The superscript $\eps$ points out this way the dependence of the solution vector on the  $\eps$-changes in the medium.
%%%%%%%%%%%%%%%%%%%%%%%%%%%%%%%%%%%%%%%
% COMMENTS ON NOTATION
%%%%%%%%%%%%%%%%%%%%%%%%%%%%%%%%%%%%%%%
We refer to \cite{Cioranescu99,Cioranescu07a} for a rigorous mathematical description of the geometry of the (periodic) porous media. Without loss of generality, we will sometimes use throughout the paper 
\begin{align*}
\int\limits_{Y}\psi(y)dy ~~\mbox{instead of}~~ \int\limits_{Y_{\rm g}}\psi_{\rm g}(y)dy + \int\limits_{Y_{\rm s}}\psi_{\rm s}(y)dy, 
\end{align*}
Similar representations also hold for counterparts defined in $\Om^\eps$. We also use the convection, $\big[\beta\big]_{\Gamma} =$ \mbox{$\beta_{\rm g} - \beta_{\rm s}$ on $\Gamma$}, to represent the jump across the boundary $\Gamma$ function that takes values in both the gas-filled region and the solid region of the domain. Furthermore, $\psi(y)\Big|_{\Gamma}$ is the restriction of the function $\psi$ on $\Gamma$. The volume measures on $\Om^\eps$ are denoted by $dx$ and $d\sigma$ are the surface measures on $\Gamma^\eps$.
%%%%%%%%%%%%%%%%%%%%%%%%%%%%%%%%%%%%%%%%%%%%%%%%%%%%%%%%%%%
% PARAMETER ESTIMATION
%%%%%%%%%%%%%%%%%%%%%%%%%%%%%%%%%%%%%%%%%%%%%%%%%%%%%%%%%%%%%
\subsection{Scaling of the mathematical model}\label{Scaling}
 Prior to performing the homogenization procedure, we normalize the system of governing equations \eqref{model1}--\eqref{model3} as discussed in \cite{Ijioma13,Ijioma15b}.  The procedure leads to a couple of important dimensionless parameters.
To obtain them, we introduce dimensionless variables as follows:
\begin{eqnarray}\label{scaled}
T^{\eps}=T^{\eps\ast}T_{\textrm c},~\C^{\eps}=\C^{\eps\ast}\C_{\textrm c},~b^{\eps}=b^{\eps\ast}b_c,~x=x^{\ast}L,~t=t_{\textrm c}t^{\ast},
\end{eqnarray}
where the subscript $\textrm c$ denotes some constant characteristic quantity while the asterisk $(\ast)$ denotes the corresponding dimensionless variable.~Similar characteristic quantities are introduced for all physical quantities entering the equations, i.e., for any generic physical quantity $\psi$, the normalization scheme is given as $\psi=\psi^{*}\psi_{\textrm c}$, where $\psi_{\textrm c}$ is some characteristic value of interest.~The dimensionless parameters are derived from these characteristic quantities and are estimated in terms of orders of magnitude in $\eps$;~see \cite[]{Auriault91}.~Various choices of the scalings usually lead to different forms of the limit problem after the asymptotic procedure as $\eps\rightarrow 0$.~Our equations in their dimensionless forms can be written as:
\begin{align}\label{dimless}
\begin{dcases}
\mathcal{P}_{T}c^{\ast}_{\rm g}\dfrac{\p T^{\eps\ast}_{\rm g}}{\p t^{\ast}} + c^{\ast}_{\rm g}\mathcal{P}e b^{\ast}\!\cdot\!\Na T^{\eps\ast}_{\rm g}- \Na\!\cdot\!(\La^{\ast}_{\rm g}\Na T^{\eps\ast}_{\rm g})=0,\\
\mathcal{P}_{T}mc^{\ast}_{\rm s}\dfrac{\p T^{\eps\ast}_{\rm s}}{\p t^{\ast}} - \mathcal{K}\Na\!\cdot\!(\La^{\ast}_{\rm s}\Na T^{\eps\ast}_{\rm s})=0,\\
\dfrac{\p \C^{\eps\ast}}{\p t^{\ast}} + \mathcal{P}e b^{\ast}\!\cdot\!\Na \C^{\eps\ast}- \mathcal{L}e^{-1}_{\rm g}\Na\!\cdot\!(D^{\ast}\Na \C^{\eps\ast})=0,\\
\La_{\rm g}^{\ast}\Na T^{\eps\ast}\!\cdot  n = \mathcal{K}\La_{\rm s}^{\ast}\Na T^{\eps\ast}\!\cdot  n + \mathcal{D}aQ^{\ast}W^{\eps\ast},\\
T^{\eps\ast}_{\rm g}=T^{\eps\ast}_{\rm s},\\
D^{\ast}\Na \C^{\eps\ast}\!\cdot  n = -\mathcal{L}e\mathcal{D}aW^{\eps\ast},
\end{dcases}
\end{align}
where 
\begin{align}
W^{\eps\ast}=A^{\ast}\C^{\eps\ast}\exp\Big(-\dfrac{T^{\ast}_{\textrm a}}{T^{\eps\ast}}\Big).
\end{align}
This formulation introduces the following global characteristic time scales:
\begin{align*}
t_{D} := \dfrac{L^2}{D_{\textrm c}}, t_{A} := \dfrac{L}{b_{\textrm c}}, t_{\La} :=\dfrac{c_{\textrm gc}L^2}{\La_{\textrm gc}}, \mbox{ and } t_{R} :=\dfrac{L}{A_{\textrm c}},
\end{align*}
where $t_D$ is the \emph{characteristic diffusion time scale}, $t_{\textrm A}$ is the \emph{characteristic advection time scale}, $t_{\La}$ is the \emph{characteristic time of conductive transfer}, while $t_R$ is the \emph{characteristic chemical reaction time scale}.
We introduce additionally the following characteristic dimensionless numbers:
\begin{align*}
&\mathcal{P}e:=\dfrac{b_{\textrm c}L}{\alpha}=\dfrac{t_{\La}}{t_{\textrm A}} \ (\mbox{\emph{P\'eclet number}}),\\
&\mathcal{L}e :=\dfrac{\alpha}{D_{\textrm c}}=\dfrac{t_{D}}{t_{\La}} \ (\mbox{\emph{Lewis number}}),\\
&\mathcal{D}a := \dfrac{A_{\textrm c}L}{\alpha}=\dfrac{t_{\La}}{t_R} \ (\mbox{\emph{Damk\"{o}hler number}}),
\end{align*}
where $\alpha := \La_{\textrm gc}/c_{\textrm gc}$ is the thermal  diffusivity.~Other dimensionless quantities introduced in Eq.~\eqref{dimless} include $\mathcal{P}_{T}=t_{\La}/t_{\textrm c}$, the ratio of characteristic time of conductive transfer to the characteristic time scale of the observation, $m =c_{\textrm sc}/c_{\textrm gc}$ is the ratio of heat capacities, $\mathcal{K}=\La_{\textrm sc}/\La_{\textrm gc}$ is the ratio of heat conductivities.

Since the fast reaction limit of non premixed combustion is reached when the characteristic transport times of mass diffusion, thermal diffusion and convection are large in comparison to characteristic times of reactions \cite{Faeth86}, we take the time of conductive heat transfer in the subdomain, $\Omega^{\eps}_{\rm g}$, as the characteristic time of the observation at the macroscopic scale, i.e., $t_{\textrm c}=t_{\La}.$ The peculiarity in the studied problem is that we have assumed the time scale for convective transport to be small in comparison to diffusion. This assumption has no physical bearing on the problem, but it would allow for investigation of a Taylor's mediated dispersion regime. Moreover, our interest in examining the long time behaviour of the combustion process is not affected by the latter assumption.

The characteristic temperature of the combustion product is given by $T_{\textrm c}=Q_{\textrm c}C_{\textrm c}/c_{\textrm gc}$ so that $T^{\ast}_{\textrm a}=T_{\textrm a}/T_{\textrm c}$ is the dimensionless activation temperature.
To simplify the setting  when passing to the homogenization limit, we assume that the constituents have heat capacities of the same order of magnitude, i.e., $m=\Ord(1).$ Since $t_{\textrm c}=t_{\La}$, it follows that $\mathcal{P}_{T}=\Ord(1)$.~The Lewis number is considered in a regime in which the time of diffusion is comparable to the time of conductive heat transfer, i.e., $\mathcal{L}e=\Ord(1)$. We study a scenario in which the constituent conductivities are of the same order of magnitude, i.e., $\mathcal{K}=\Ord(1)$. 

%%%%%%%%%%%%%%%%%%%%%%%%%%%%%%%%%%%%%%%%%%%%%%%%%%%%%%%%%%%%%%%%%%%%%%%%%%%%%%
\subsection{The microscopic problem}\label{micro-pb}
%%%%%%%%%%%%%%%%%%%%%%%%%%%%%%%%%%%%%%%%%%%%%%%%%%%%%%%%%%%%%%%%%%%%%%%%%%%%%%
The combustion regime of interest is essentially flaming, which is characterized by high Damk\"{o}hler numbers. Thus, we choose the estimate $\mathcal{D}a = \Ord(\eps^{-1})$ in the our model. In order to study effects similar to a Taylor's mediated dispersion regime, the convective transport is estimated as $\mathcal{P}e = \Ord(\eps^{-1})$. Note that, in principle, the asymptotic homogenization procedure can  be performed at least formally for other combinations of scalings in terms of $\eps$.  However, when the P\'eclet and Damk\"{o}hler numbers are not balanced, then the validity of our starting model (\ref{dimless}) is restricted. For instance, high P\'eclet numbers in combination with moderate Damk\"{o}hler numbers leads to turbulence regimes (compare \cite{Sivashinsky83}), while high Damk\"{o}hler numbers with moderate P\'eclet numbers most likely facilitate the occurrence of flames. In both such cases, the microscopic model has to be changed essentially. In one case, it becomes a system of stochastic partial differential equations, while in the other case the model stays at the deterministic level of description, but must include information about the {\em a priori} unknown position, shape and velocity of the flames. Taking into account the parameter estimates above and considering all the other parameters of the order of $\Ord(1)$ with respect to $\eps$, we rewrite \eqref{dimless} after dropping the ($\ast$) and we obtain the following microscopic problem: 
\begin{align}\label{eqmicro}
\begin{dcases}
c_{\rm g}\frac{\p T^{\eps}_{\rm g}}{\p t} + c_{\rm g}\dfrac{ b^\eps}{\eps}\!\cdot\!\Na T^{\eps}_{\rm g}- \Na\!\cdot\!(\La_{\rm g}\Na T^{\eps}_{\rm g})=0,&\mbox{in $S\times \Om^{\eps}_{\rm g},$}\\
c_{\rm s}\dfrac{\p T^{\eps}_{\rm s}}{\p t} -\Na\!\cdot\!(\La_{\rm s}\Na T^{\eps}_{\rm s})=0, &\mbox{in $S\times \Om^{\eps}_{\rm s},$}\\
\frac{\p \C^{\eps}}{\p t} + \dfrac{ b^\eps}{\eps}\!\cdot\!\Na \C^{\eps}- \Na\!\cdot\!(D\Na \C^{\eps})=0,&\mbox{in $S\times \Om^{\eps}_{\rm g},$}\\
\begin{rcases*}
\La_{\rm g}\Na T^{\eps}\!\cdot  n = \La_{\rm s}\Na T^{\eps}\!\cdot  n + Q\dfrac{W^{\eps}}{\eps},\\
T^{\eps}_{\rm g}=T^{\eps}_{\rm s},\\
D\Na \C^{\eps}\!\cdot  n = -\dfrac{W^{\eps}}{\eps},
\end{rcases*} &\mbox{in $S\times \Gamma^{\eps}.$}\\
\C^\eps(x,0)=\C_{0}(x) \mbox{ for } x\in \Omega_{\rm g}^\eps,\\
T^\eps(0,x)=T_{0}(x) \mbox{ for }  x\in \Omega^\eps, 
\end{dcases}
\end{align}
where $S=(0,t_f)$ is the time interval of observation of the combustion process with $t_f$, some final time. We refer to this system of equations as $\mathcal{P}^\eps$.
\begin{remark}
As expected, performing in the dimensional model (\ref{dimless}) the change of variables $x=\eps y$ and $t=\eps^2 \tau$ leads to the same parabolic scaling in $\eps$ as shown in (\ref{eqmicro}). 
\end{remark}

\begin{remark}
$\mathcal{P}^\eps$ is characterized by the presence of large P\'eclet numbers, which may trigger large drifts in the concentration and temperature profiles. In particular, fast drift appears as a form of response at the observation scale due to the dominant convection at the microscopic level. As will be demonstrated later in the text, the strong convection regime at the microscopic level is characterized by the structure of the ensuing cell problems of the homogenization procedure. 
\end{remark}

%%%%%%%%%%%%%%%%%%%%%%%%%%%%%%%%%%%%%%%%%%%
% RESTRICTION ON DATA, FUNCTIONS AND PARAMETERS
%%%%%%%%%%%%%%%%%%%%%%%%%%%%%%%%%%%%%%%%%%%
\section{Assumptions.  Concept of solution. Technical preliminaries}
We take into account the following assumptions on the involved functions, physical quantities, and parameters:
\begin{itemize}
%\item[(H1)] $c_{\rm g}, c_{\rm s}, \lambda_{\rm g},\lambda_{\rm s},D, Q, T_a, A\in (0,\infty)$;
\item[(H1)] The velocity of the gaseous mixture is periodic, incompressible and it does not penetrate the solid, i.e., $b(y) \in L^{\infty}_{\#}(Y_{\rm g}; \mathbb{R}^d)$ is such that
\begin{eqnarray}
\textrm{div}_yb(y) = 0\quad \mbox{in $Y_{\rm g},$}\quad b(y)\cdot n(y) = 0\quad \mbox{in $\Gamma,$}
\end{eqnarray}
where $n(y)$ is the outward unit normal to $Y_{\rm g}.$ 

\item[(H2)] The molecular diffusion coefficient is periodic, isotropic (tensor) and restricted to the gas region, i.e., $D(y) \in L^{\infty}_{\#}(Y_{\rm g})$. In addition, it satisfies the uniformly coercivity property, i.e., there exists a constant $C>0$ such that for any $\xi \in \Real^d$ it holds
$$D(y)\xi\cdot\xi \ge C|\xi|^2 \mbox{ a.e. in $Y_{\rm g}.$ }$$

\item[(H3)] The thermal conductivity coefficient is periodic, isotropic (tensor) and it is defined in both the gas and the solid regions, i.e., $\lambda(y) \in L^{\infty}_{\#}(Y)$ and
\begin{align}
\lambda(y) = 
\begin{dcases}
\lambda_{\rm g}, &\mbox{in $Y_{\rm g}$,}\\
\lambda_{\rm s}, &\mbox{in $Y_{\rm s}.$}
\end{dcases}
\end{align}
In addition, it satisfies the uniformly coercivity property, i.e., there exists a constant $C>0$ such that for any $\xi \in \Real^d$ it holds
$$\lambda(y)\xi\cdot\xi \ge C|\xi|^2 \mbox{ a.e. in $Y.$ }$$
\item[(H4)] The heat capacity is periodic and it is defined in the gas and the solid regions, i.e.,  $c(y) \in L^{\infty}_{\#}(Y)$ and
\begin{align}
c(y) = 
\begin{dcases}
c_{\rm g}, &\mbox{in $Y_{\rm g}$,}\\
c_{\rm s}, &\mbox{in $Y_{\rm s}.$}
\end{dcases}
\end{align}

\item[(H5)] The nonlinear function \mbox{$f:\Real \rightarrow \Real$} is given such that $f(\alpha)$ is positive and bounded for $\alpha\in (0, T_b)$, where $T_b > T_a$ and $T_b$ is the combustion temperature and $f=0,$ for $\alpha \leq 0$. A direct consequence of $(H5)$ is the following exponentiation property:
\begin{align}\label{mono}
f \geq f^2 \geq \ldots \geq f^n,~\mbox{for all $n\in \mathbb{N}.$}
\end{align}
\begin{figure}[!h]
\centering
\includegraphics[scale=0.35]{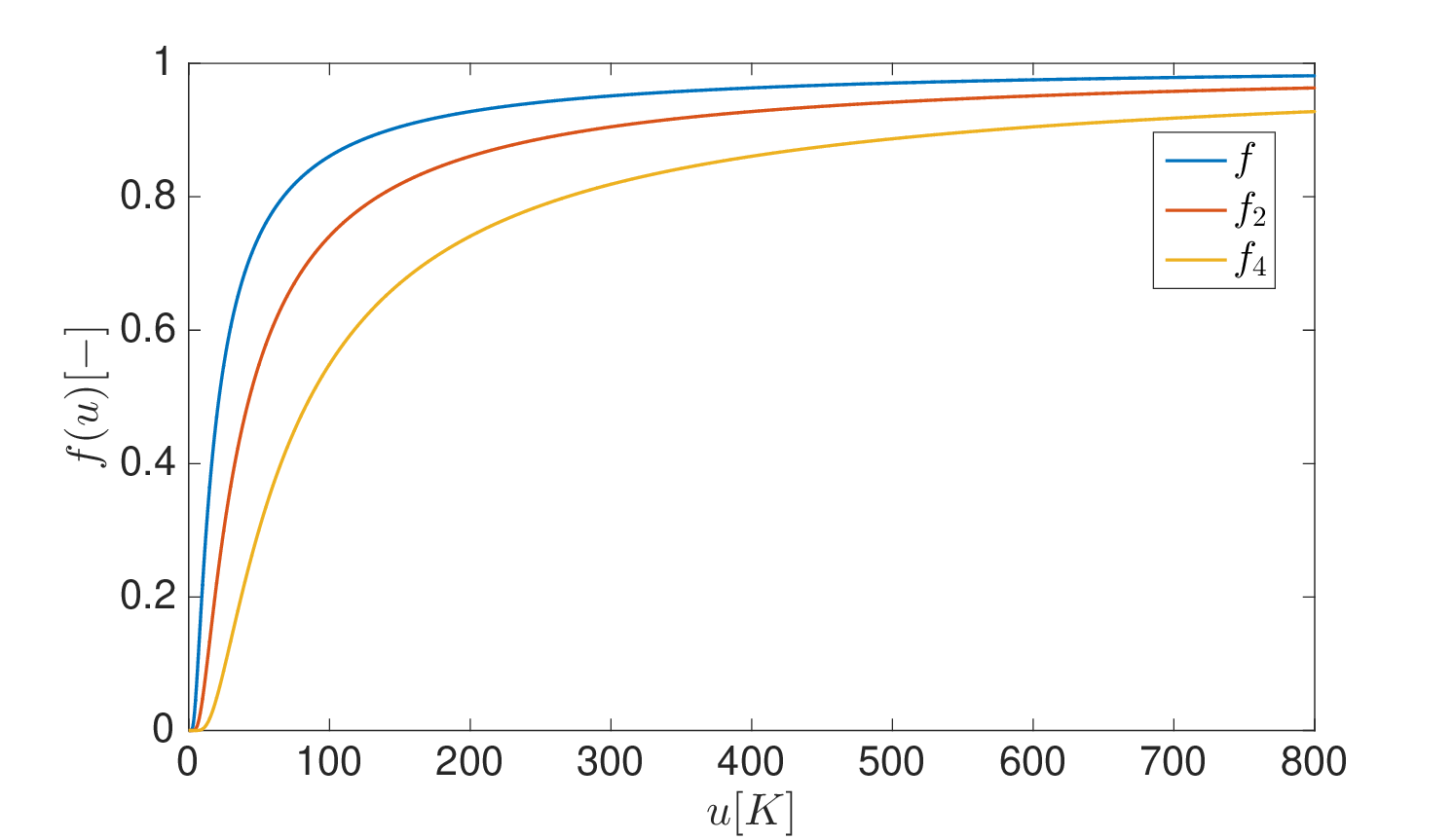}
\caption{The profile of the function $f$ and its exponentiation property illustrated for $n=2,4$. $T_b=800 K$ and $T_a=15 K.$}
\end{figure}
\item[(H6)] The initial data are non-negative, i.e., $\C_0\geq 0, T_0 \geq 0$ and $\C_0\in L^2(\mathbb{R}^d)\cap L^\infty(\Real^d)$, $T_{0}\in  L^2(\mathbb{R}^d)\cap L^\infty(\Real^d)$.
\end{itemize}
\begin{remark} Assumptions (H2)--(H4) are physically motivated in the sense that heat capacities, heat conductivities and diffusion coefficients for isotropic materials, as described in the different regions, are strictly positive real constants, i.e., $c_{\rm g}, c_{\rm s}, \lambda_{\rm g},\lambda_{\rm s}, D\in (0,\infty).$ Depending on the material, the quantities $\lambda_{\rm g},\lambda_{\rm s}$ and $D$ could be anisotropic and hence tensors. $(H3)$ is assumed based on the initial state of the process. On the other hand, $(H1)$ is a mathematical restriction which strongly delimits the applicability of the method. For instance, note that if the velocity field $b^\eps$ would come from a background fluid satisfying the compressible Navier-Stokes equations, then it would automatically depend on both variables $x$ and $x/\eps$, and hence this approach would not be applicable anymore in a straightforward fashion; see also Remark 2 in \cite{Mik}.  Assumption $(H5)$ is a form of the non-linear kinetics generally accepted in the combustion engineering community. 
 \end{remark}
%%%%%%%%%%%%%%%%%%%%%%%%%%%%%%%%%%%%%%%%%%%
 % WEAK FORMULATION
 %%%%%%%%%%%%%%%%%%%%%%%%%%%%%%%%%%%%%%%%%%%
 We introduce the function space
 $$V^\eps:=\{(\phi,\psi)\in H^1(\Omega^\eps)\times H^1(\Omega_{\rm g}^\eps)\}.$$ 
\begin{definition}\label{X1} Assuming $(H5)$, the couple $(T^\eps,\C^\eps)\in L^2(S;H^1(\Omega^\eps))\times L^2(S; H^1(\Omega_{\rm g}^\eps))$ is called a weak solution to $\mathcal{P}^\eps$ if for all $t\in S$ and $(T^\eps(\cdot,0),\C^\eps(\cdot,0))=(T_{0},\C_{0})$,  the following identities hold
\begin{align}\label{Teq1}
&\int\limits_{\Om^\eps}c^\eps\dfrac{\p T^\eps}{\p t}\phi~dx  + \dfrac{1}{\eps}\int\limits_{\Om^\eps_{\rm g}}c_{\rm g}b^\eps\cdot\Na T^\eps\phi~dx+\int\limits_{\Om^\eps}\La^\eps\Na T^\eps\cdot \Na\phi~dx-\dfrac{Q}{\eps}\int\limits_{\Gamma^\eps}W^\eps\phi d\sigma = 0,
\end{align}
 \begin{equation}\label{Ceq1}
\int\limits_{\Om^\eps_{\rm g}}\dfrac{\p \C^\eps}{\p t}\psi  dx  + \dfrac{1}{\eps}\int\limits_{\Om^\eps_{\rm g}}b^\eps\cdot\Na \C^\eps\Na\psi dx + \int\limits_{\Om^\eps_{\rm g}}D^\eps\Na \C^\eps\cdot \Na\psi~dx+\dfrac{1}{\eps}\int\limits_{\Gamma^\eps}W^\eps\psi d\sigma=0
\end{equation}
for all test functions $(\phi,\psi)\in V^\eps$.
\end{definition}
%%%%%%%%%%%%%%%%%%%%%%%%%%%%%%%%%%%%%%%%%%%%
% A PRIORI BOUNDS
%%%%%%%%%%%%%%%%%%%%%%%%%%%%%%%%%%%%%%%%%%%%
\begin{lemma}\label{apriori}
Let $(T^\eps,\C^\eps)$ be a weak solution to $\mathcal{P}^\eps$ in the sense of Definition \ref{X1}. Assume $(H1)$--$(H6)$. There exists a constant $C>0$, which is independent of the choice of $\eps$, such that the following a priori bound is satisfied
 \begin{eqnarray}\label{Apeq0}
||T^\eps||_{L^\infty(S;L^2(\Om^\eps))} + ||\C^\eps||_{L^\infty(S;L^2(\Om^\eps_{\rm g}))} + ||\Na T^\eps||_{L^2(S\times \Om^\eps)} + ||\Na \C^\eps||_{L^2(S\times \Om^\eps_{\rm g})}\\\nonumber
+ \sqrt{\eps}||\W^\eps||_{L^\infty(S;L^2(\Gamma^\eps))} \leq C\left(||T_0||_{L^2(\Real^d)} + ||\C_0||_{L^2(\Real^d)}\right),
  \end{eqnarray}
where $\W^\eps = \dfrac{1}{\eps}\C^\eps f(T^\eps).$
\end{lemma}

\begin{proof}
The desired a priori bound can be obtained by testing \eqref{Teq1} and \eqref{Ceq1} by $(\phi,\psi)=(T^\eps, Q\C^\eps)$ and summing up the resulting expressions. Starting off with \eqref{Ceq1}, we set $\psi=Q\C^\eps$ and get
%%%%%%%%%%%%%%%%%%%%%
\begin{align}\label{Apeq1}
& \dfrac{Q}{2}\dfrac{d}{dt}||\C^\eps(t)||^2_{L^2(\Om^\eps_{\rm g})} +  \dfrac{Q}{\eps}( b^{\eps}\Na \C^\eps,\C^\eps)  + \dfrac{Q}{2}||\Na \C^\eps(t)||^2_{L^2(\Om^\eps_{\rm g})}\\\nonumber
&+ \dfrac{QA}{\eps}\int\limits_{\Gamma^\eps}(\C^\eps)^2 f(T^\eps)~d\sigma =0.
\end{align} 
%%%%%%%%%%%%%%%%%%%%%
The convection term vanishes since by $(H1)$ we have that 
 \begin{align*}
 \int\limits_{\Omega_{\rm g}^\eps} b^\eps\cdot\Na \C^\eps \C^\eps dx=-\frac{1}{2}\int\limits_{\Omega_{\rm g}^\eps} {\rm div}\left( b^\eps |\C^\eps|^2\right)dx=-\frac{1}{2} \int\limits_{\Gamma^\eps}|C^\eps |^2  b^\eps\cdot  n d\sigma=0.
 \end{align*}
By $(H5)$, we apply the property $f \geq f^2$ on the nonlinear term in \eqref{Apeq1}, i.e.,
\begin{align}
 \dfrac{QA}{\eps}\int\limits_{\Gamma^\eps}(\C^\eps)^2 f(T^\eps)~d\sigma \geq \dfrac{QA}{\eps}\int\limits_{\Gamma^\eps}|\C^\eps f(T^\eps)|^2~d\sigma=\eps C \int\limits_{\Gamma^\eps}\Big|\dfrac{1}{\eps}\C^\eps f(T^\eps)\Big|^2d\sigma.
\end{align}
Integrating \eqref{Apeq1} with respect to time leads to the following
\begin{align}\label{Apeq2}
&||\C^\eps(t)||^2_{L^{\infty}(S;L^2(\Om^\eps_{\rm g}))} + ||\Na \C^\eps(t)||^2_{L^2(S\times\Om^\eps_{\rm g})}
+\eps \Big\|\dfrac{1}{\eps}\C^\eps f(T^\eps)\Big\|^2_{L^2(S\times \Gamma^\eps)}\\\nonumber
& \leq C\|C_0\|^2_{L^2(\Real^2)}.
\end{align}
Next, taking $\phi=T^\eps$ in \eqref{Teq1}, we obtain  
\begin{align}\label{Apeq3}
&\dfrac{1}{2}\dfrac{d}{dt}||T^\eps(t)||^2_{L^2(\Om^\eps)} +  \dfrac{1}{\eps}(c_{\rm g} b^{\eps}\Na T^\eps,T^\eps) + ||\Na T^\eps(t)||^2_{L^2(\Om^\eps)}\\
&\leq\dfrac{QA}{\eps}\int\limits_{\Gamma^\eps}\C^\eps f(T^\eps)T^\eps~d\sigma.\nonumber
\end{align}
As in \eqref{Ceq1}, the convection term vanishes due to (H1). The nonlinear term on the right hand side of \eqref{Apeq3} can be estimated as follows:
 \begin{align}\label{Apeq4}
 \dfrac{QA}{\eps}\int\limits_{\Gamma^\eps}\C^\eps f(T^\eps)T^\eps~d\sigma dt \leq \dfrac{\delta}{2\eps}\int\limits_{\Gamma^\eps}|T^\eps|^2~d\sigma + \dfrac{\eps}{2\delta^\prime}\int\limits_{\Gamma^\eps}|\W^\eps|^2~d\sigma,
 \end{align}
 where $\W^\eps = \eps^{-1} \C^\eps f(T^\eps)$ and $\delta=\delta(\delta^\prime, Q, A) >0, \delta^\prime >0.$ The last surface integral in \eqref{Apeq4} is bounded due to \eqref{Apeq2}. To estimate the first surface integral in \eqref{Apeq4}, we choose $\delta=\eps^2$ and apply the trace inequality for $\eps$-dependent hypersurfaces (see, e.g., Lemma 3 of \cite{Hornung91} and Lemma 2.7 of \cite{Sara}):
\begin{align}\label{Apeq5}
\eps \int\limits_{\Gamma^\eps}|T^\eps|^2~d\sigma \leq C\int\limits_{\Om^\eps}(|T^\eps|^2 + \eps^2|\Na T^\eps|^2)~dx.
\end{align}
Integrating \eqref{Apeq3}-\eqref{Apeq5} with respect to time leads to
%%%%%%%%%%%%%%%%%%%%%%%%%%%%%%%%%%%%%%%%%%%%%%%%%%%%%%%%
\begin{align}\label{Apeq6}
C\|T^\eps\|^2_{L^\infty(S;L^2(\Om^\eps))} + (1-\eps^2C)\|\Na T^\eps(t)\|^2_{L^2(S\times \Om^\eps)} \leq C\|T_0\|^2_{L2(\Real^2)}.
\end{align}
%%%%%%%%%%%%%%%%%%%%%%%%%%%%%%%%%%%%%%%%%%%%%%%%%%%%%%%%
By choosing $\eps$ small enough and combining \eqref{Apeq2} and \eqref{Apeq6}, we deduce \eqref{Apeq0}.
\end{proof}

\subsection{Compactness in the moving coordinate frame}
%%%%%%%%%%%%%%%%%%%%%%%%%%%%%%%%%%
\paragraph{Restriction to a compact subset of $\Real^d$}
%%%%%%%%%%%%%%%%%%%%%%%%%%%%%%%%%%
Since the problem is posed in an unbounded domain, Rellich Theorem does not apply, and hence there is no hope of establishing the compactness of the sequence of functions $(T^\eps,\C^\eps)$. Thus, we first restrict the problem to a compact subset of $\Real^d$. However, due to the presence of the dominant convection in the problem, nontrivial convergence of the family of solutions $T^\eps(t,x)$ and $\C^\eps(t,x)$ is not viable in fixed coordinates $x$.~Thus, we
will follow the argument in \cite{Harsha,Maro} and prove compactness results in the moving coordinates  $\big(x + \eps^{-1}b_T, x + \eps^{-1}b_\C\big).$ To work in the latter frame of reference, we will introduce functions in the new reference frame, following a similar notation as described in \cite{Harsha}:
For any given function $u^\eps(t,x)$ in the fixed coordinate frame, we define its counterpart in the moving coordinates as
\begin{align}\label{mov1}
\hat{u}^\eps(t,x) = u^\eps\Big(t,x+\dfrac{b^*}{\eps}t\Big),
\end{align}
where $b^*$ is a constant vector representing the average velocities, $b_T$ and $b_\C$, of the temperature and concentration fields respectively. For drifts moving in the opposite direction to the direction of \eqref{mov1}, (cf. Section \ref{rigor} for the definition of the two-scale convergence with drift.), we will use
\begin{align}\label{mov2}
\check{u}^\eps(t,x) = u^\eps\Big(t,x-\dfrac{b^*}{\eps}t\Big).
\end{align}
Additionally, a consequence of \eqref{mov1} (respectively \eqref{mov2}) is to characterize the mass conservation of the porous medium in the moving coordinates. For that, we assume the motion of the domain  depends on the drift of the governing physical process, i.e., for the thermal dispersion in the porous medium, we write
\begin{align}\label{mov3}
\hat\Om^\eps= \Big\{x + \dfrac{b_T}{\eps}t ~\Big| ~x\in \Om^\eps\Big\},
\end{align}
and for mass dispersion in the gaseous region of the porous medium, we will use the definition
\begin{align}\label{mov4}
\hat\Om^\eps_{\rm g}= \Big\{x + \dfrac{b_\C}{\eps}t ~\Big| ~x\in \Om^\eps_{\rm g}\Big\}.
\end{align}
\begin{remark}
For the case of filtration combustion of gases, i.e., without a thermal spread in the solid region, the disparity in the moving coordinates, given by \eqref{mov1}--\eqref{mov4}, no longer applies since in such a case it holds that $b_T=b_\C.$
\end{remark}
%%%%%%%%%%%%%%%%%%%%%%%
% LEMMA 2.2
%%%%%%%%%%%%%%%%%%%%%%%
\begin{lemma}\label{restEq1}
Let $(T^\eps, \C^\eps)$ be the solution of $\mathcal{P}^\eps$. Then, for any $\delta>0$, there exists $R(\delta)>0$ such that for all $t\in \bar{S}$
\begin{eqnarray}
\big\|\widehat{T}^\eps(t,x)\big\|_{L^2(\Om^\eps(t)\cap (\Real^d\setminus Q_{R(\delta)}))}\leq \delta,~~\big\|\widehat{\C}^\eps(t,x)\big\|_{L^2(\Om_{\rm g}^\eps(t)\cap (\Real^d\setminus Q_{R(\delta)}))}\leq \delta,
\end{eqnarray}
where 
\begin{align}\label{cube}
Q_{R(\delta)}=\big]-R(\delta), R(\delta)\big[^d \in \Real^d.
\end{align}
\end{lemma}

\begin{proof}
The proof goes in a similar manner as in \cite{Harsha}. We adapt their argument for our purpose. We introduce a smooth cut-off function $\psi \in C^{\infty}(\Real)$ satisfying $0 \leq \psi(r) \leq 1, \psi = 0$ for $r\leq 1, \psi = 1$ for $r \geq 2.$~Then, for $x\in \Real^d,$ we denote $\psi_R(x)=\psi(|x|/R)$ and choose as test functions $\big(T^\eps\check{\psi}_{R,T}, Q\C^\eps\check{\psi}_{R,\C}\big)$, in the variational formulation \eqref{Teq1}-\eqref{Ceq1}, where 
$$
\check{\psi}_{R,T}(t,x) = \psi_R\left(x-\dfrac{b_T}{\eps}t\right),~\check{\psi}_{R,\C}(t,x) = \psi_R\left(x-\dfrac{b_\C}{\eps}t\right).
$$
Next, we integrate by parts in time the terms with the time derivatives to obtain
\begin{align*}
&\int\limits_0^t\!\int\limits_{\Om^\eps}c^\eps\dfrac{\p T^\eps}{\p t}T^\eps\check{\psi}_{R,T}~dxds = \dfrac{1}{2\eps}\int\limits_0^t\!\int\limits_{\Om^\eps}c^\eps|T^\eps|^2b_T\cdot\Na\check{\psi}_{R,T}~dxds\\
&+\int\limits_{\Om^\eps}c^\eps|T^{\eps}(t,x)|^2\check{\psi}_{R,T}(t,x)~dx-\int\limits_{\Om^\eps}c^\eps|T_0|^2\check{\psi}_{R}(x)~dx
\end{align*}
and
\begin{align*}
&Q\int\limits_0^t\!\int\limits_{\Om^\eps_{\rm g}}\dfrac{\p \C^\eps}{\p t}\C^\eps\check{\psi}_{R,\C}~dxds = \dfrac{Q}{2\eps}\int\limits_0^t\!\int\limits_{\Om^\eps_{\rm g}}|\C^\eps|^2b_\C\cdot\Na\check{\psi}_{R,\C}~dxds\\
&+Q\int\limits_{\Om^\eps_{\rm g}}|\C^{\eps}(t,x)|^2\check{\psi}_{R,\C}(t,x)~dx-Q\int\limits_{\Om^\eps_{\rm g}}|\C_0|^2\check{\psi}_{R}(x)~dx.
\end{align*}
In turn, we apply integration by parts on the terms with the space derivatives. Starting with the convection terms, we obtain
\begin{align*}
\dfrac{1}{\eps}\int\limits_0^t\!\int\limits_{\Om^\eps}c_{\rm g}b^\eps\cdot\Na T^\eps T^\eps \check{\psi}_{R,T}~dxds = -\dfrac{1}{2\eps}\int\limits_0^t\!\int\limits_{\Om^\eps}c_{\rm g}|T^\eps|^2b^\eps\cdot\Na \check{\psi}_{R,T}~dxds
\end{align*}
and
\begin{align*}
\dfrac{Q}{\eps}\int\limits_0^t\!\int\limits_{\Om^\eps_{\rm g}}b^\eps\cdot\Na \C^\eps \C^\eps \check{\psi}_{R,\C}~dxds = -\dfrac{Q}{2\eps}\int\limits_0^t\!\int\limits_{\Om^\eps_{\rm g}}|\C^\eps|^2b^\eps\cdot\Na \check{\psi}_{R,\C}~dxds
\end{align*}
and then the diffusive terms 
\begin{align*}
&-\int\limits_0^t\!\int\limits_{\Om^\eps}{\rm div}(\La^\eps\Na T^\eps)T^\eps\check{\psi}_{R,T}~dxds = \int\limits_0^t\!\int\limits_{\Om^\eps}\La^\eps\Na T^\eps\cdot\Na T^\eps\check{\psi}_{R,T}~dxds\\
&+\dfrac{1}{2}\int\limits_0^t\!\int\limits_{\Om^\eps}\La^\eps\Na|T^\eps|^2\Na\check{\psi}_{R,T}~dxds
- \dfrac{QA}{\eps}\int\limits_0^t\!\int\limits_{\Gamma^\eps}\C^\eps T^\eps f(T^\eps)\check{\psi}_{R,T}~d\sigma ds
\end{align*}
and
\begin{align*}
&-Q\int\limits_0^t\!\int\limits_{\Om^\eps_{\rm g}}{\rm div}(D^\eps\Na \C^\eps)\C^\eps\check{\psi}_{R,\C}~dxds = Q\int\limits_0^t\!\int\limits_{\Om^\eps_{\rm g}}D\Na \C^\eps\cdot\Na \C^\eps\check{\psi}_{R,\C}dxds\\
&+\dfrac{Q}{2}\int\limits_0^t\!\int\limits_{\Om^\eps_{\rm g}}D^\eps\Na|\C^\eps|^2\Na\check{\psi}_{R,\C}~dxds
+ \dfrac{QA}{\eps}\int\limits_0^t\!\int\limits_{\Gamma^\eps}|\C^\eps|^2 f(T^\eps)\check{\psi}_{R,\C}~d\sigma ds.
\end{align*}
%%%%%%%%%%%%%%%%%%%%%%%%%%%%%%%%%%%%%%%%%%%%%%%%% PAGE 14
The surface integrals can be reduced to an energy-like form as follows:
\begin{align*}
&\eps QA \int\limits_0^t\!\int\limits_{\Gamma^\eps}\Big|\dfrac{1}{\eps}\C^\eps f(T^\eps)\Big|^2\check{\psi}_{R,\C}~d\sigma=\dfrac{QA}{\eps}\int\limits_0^t\!\int\limits_{\Gamma^\eps}|\C^\eps f(T^\eps)|^2\check{\psi}_{R,\C}~d\sigma ds\\
&\leq \dfrac{QA}{\eps}\int\limits_0^t\!\int\limits_{\Gamma^\eps}(\C^\eps)^2 f(T^\eps)\check{\psi}_{R,\C}~d\sigma ds
\end{align*}
since $f^2(T) \leq f(T)$ by $(H5)$. The next surface integral can be estimated as follows:
\begin{align}\label{EstEq1}
\dfrac{QA}{\eps}\int\limits_0^t\!\int\limits_{\Gamma^\eps}\C^\eps T^\eps f(T^\eps)\check{\psi}_{R,T}~d\sigma ds \leq \dfrac{C_1\xi}{\eps}\int\limits_0^t\!\int\limits_{\Gamma^\eps}|T^\eps|^2~d\sigma ds + \eps C_2 \int\limits_0^t\!\int\limits_{\Gamma^\eps}|\W^\eps|^2~d\sigma ds,
\end{align}
where $\W^\eps = \eps^{-1}\C^\eps f(T^\eps), C_1=C(Q, A, ||\check{\psi}_{R,T}||_{L^\infty(\Om^\eps)})$ and $C_2=C\bigg(\dfrac{1}{\xi}, ||\check{\psi}_{R,T}||_{L^\infty(\Om^\eps)}\bigg)$, with $\xi >0.$ The last integral in \eqref{EstEq1} containing $\W^\eps$ is bounded by virtue of Lemma \ref{apriori}. By choosing $\xi=\eps^2$, the first surface integral in \eqref{EstEq1} reduces by the trace inequality \cite{Hornung91}, to
\begin{align}\label{EstEq2}
\eps C_1 \int\limits_0^t\!\int\limits_{\Gamma^\eps}|T^\eps|^2~d\sigma ds \leq C\big(||\check{\psi}_{R,T}||_{L^\infty(\Om^\eps)}\big) \int\limits_0^t\!\int\limits_{\Om^\eps}\big(|T^\eps|^2 + \eps^2|\Na T^\eps|^2\big)~dx ds.
\end{align}
Note that the estimates in \eqref{EstEq1} and \eqref{EstEq2} tend to zero since by the definition of $\psi_{R,T}, ||\check{\psi}_{R,T}||_{L^\infty(\Om^\eps)}\rightarrow 0$ as $R$ tends to $\infty$. 
Combining all the terms together, we have
\begin{align}
%\begin{split}
&\int\limits_{\Om^\eps}c^\eps|T^{\eps}(t,x)|^2\check{\psi}_{R,T}(t,x)~dx + \int\limits_0^t\!\int\limits_{\Om^\eps}\La^\eps\Na T^\eps\cdot\Na T^\eps\check{\psi}_{R,T}~dxds \\
&+Q\int\limits_{\Om^\eps_{\rm g}}|\C^{\eps}(t,x)|^2\check{\psi}_{R,\C}(t,x)~dx\\
&+ Q\int\limits_0^t\!\int\limits_{\Om^\eps_{\rm g}}D^\eps\Na \C^\eps\cdot\Na \C^\eps\check{\psi}_{R,\C}~dxds + \eps QA \int\limits_0^t\!\int\limits_{\Gamma^\eps}\Big|\dfrac{1}{\eps}\C^\eps f(T^\eps)\Big|^2\check{\psi}_{R,\C}~d\sigma\\
\label{convEq1}
& \leq \dfrac{1}{2\eps}\int\limits_0^t\!\int\limits_{\Om^\eps}c^\eps|T^\eps|^2(b^\eps-b_T)\cdot\Na\check{\psi}_{R,T}~dxds + \dfrac{Q}{2\eps}\int\limits_0^t\!\int\limits_{\Om^\eps_{\rm g}}|\C^\eps|^2(b^\eps-b_\C)\cdot\Na\check{\psi}_{R,\C}~dxds\\
\label{diffEq1}
&-\dfrac{1}{2}\int\limits_0^t\!\int\limits_{\Om^\eps}\La^\eps\Na|T^\eps|^2\Na\check{\psi}_{R,T}~dxds -\dfrac{Q}{2}\int\limits_0^t\!\int\limits_{\Om^\eps_{\rm g}}D^\eps\Na|\C^\eps|^2\Na\check{\psi}_{R,\C}~dxds\\
\label{initEq1}
&+\int\limits_{\Om^\eps}c^\eps|T_0|^2\check{\psi}_{R,T}(x)~dx + \int\limits_{\Om^\eps_{\rm g}}|\C_0|^2\check{\psi}_{R,\C}(x)~dx.
%\end{split}
\end{align}
To deal with the singular nature of the convective terms on the right hand side of \eqref{convEq1}, we introduce two auxiliary problems for the unknown quantities, $(\Pi, \Sigma)\in [H^1_{\#}(Y)/\Real]^d \times [H^1_{\#}(Y_{\rm g})/\Real]^d$:
\begin{align}\label{aux1}
\begin{dcases}
-\Delta \Pi_i(y) = c(y)(b_{T,i}-b_{i}(y)) &\mbox{in $Y,$}\\
-\Na\Pi_i\cdot  n = 0 &\mbox{on $\Gamma,$}\\
\mbox{$\Pi_i(y)$ is $Y$-periodic}
\end{dcases}~~
\begin{dcases}
-\Delta \Sigma_i(y) = b_{\C,i}-b_{i}(y) &\mbox{in $Y_{\rm g},$}\\
-\Na\Sigma_i\cdot  n = 0 &\mbox{on $\Gamma,$}\\
\mbox{$\Sigma_i(y)$ is $Y$-periodic,}
\end{dcases}
\end{align}
where the solutions $\Pi_i$ and $\Sigma_i$ are unique up to additive constants, by virtue of the expressions \eqref{drift1} and \eqref{drift2} defining the effective drift constants. Since $\eps\Na \Pi^\eps_i(x) = (\Na_y\Pi_i)(x/\eps)$ (similarly for the variable $\Sigma$), periodic extensions of \eqref{aux1} in $\Real^d$ allow us to write:
\begin{align}\label{aux2}
\begin{dcases}
-\eps^2\Delta\Pi^\eps_i(x) = c^\eps(x)(b_{T,i}-b^\eps_{i}(x)) &\mbox{in $\Om^\eps,$}\\
-\Na\Pi^\eps_i\cdot  n = 0 &\mbox{on $\Gamma^\eps,$}\\
\mbox{$\Pi^\eps_i(x)$ is $\eps$-periodic}
\end{dcases}~~
\begin{dcases}
-\eps^2\Delta \Sigma^\eps_i(x) = b_{\C,i}-b^\eps_{i}(x) &\mbox{in $\Om^\eps_{\rm g},$}\\
-\Na\Sigma^\eps_i\cdot  n = 0 &\mbox{on $\Gamma^\eps,$}\\
\mbox{$\Sigma^\eps_i(x)$ is $\eps$-periodic.}
\end{dcases}
\end{align}
%%%%%%%%%%%%%%%%%%%%%%%
\begin{remark}
For brevity of presentation, we have assumed suitable extensions for the velocity field and parameters defined in $Y$, respectively in $\Om^\eps$. Specifically, in situations as in \eqref{aux1} and \eqref{aux2}, we introduce distinct problems in the solid and gas-filled regions of $Y$ (respectively $\Om^\eps$) as follows:
\begin{align}\label{aux1Remark}
\begin{dcases}
-\Delta \Pi_{g,i}(y) = c_{\rm g}(b_{T,i}-b_{i}(y)) &\mbox{in $Y_{\rm g},$}\\
-\Delta \Pi_{s,i}(y) = c_{\rm s}b_{T,i}              &\mbox{in $Y_{\rm s},$}\\
(\Na\Pi_{g,i}-\Na\Pi_{s,i})\cdot  n = 0 &\mbox{on $\Gamma,$}\\
\mbox{$\Pi_i(y)$ is $Y$-periodic.}
\end{dcases}~
\begin{dcases}
-\eps^2\Delta\Pi^\eps_i(x) = c_{\rm g}(b_{T,i}-b^\eps_{i}(x)) &\mbox{in $\Om^\eps_{\rm g},$}\\
-\eps^2\Delta\Pi^\eps_i(x) = c_{\rm s}b_{T,i} &\mbox{in $\Om^\eps_{\rm s},$}\\
(\Na\Pi^\eps_{\textrm g,i}-\Na\Pi^\eps_{\textrm s,i})\cdot  n = 0 &\mbox{on $\Gamma^\eps,$}\\
\mbox{$\Pi^\eps_i(x)$ is $\eps$-periodic.}
\end{dcases}
\end{align}
\end{remark}
Substituting \eqref{aux2} in \eqref{convEq1} and integrating by part
\begin{align*}
\sum_{i=1}^{d}\int\limits^t_0\Bigg(\int\limits_{\Om^\eps}\eps\Na\Pi_i^\eps\cdot \Na\Big(|T^\eps|^2\p_{x_i}\check{\psi}_{R,T}\Big)~dx +\int\limits_{\Om^\eps_{\rm g}}\eps\Na\Sigma_i^\eps\cdot \Na\Big(|\C^\eps|^2\p_{x_i}\check{\psi}_{R,\C}\Big)~dx \Bigg)ds.
\end{align*}
By the definition of $\psi_{R,\alpha}$, we have $\norm{\Na\check{\psi}_{R,T}}_{L^\infty(\Om^\eps)} \leq C/R$ and $\norm{\Na\check{\psi}_{R,\C}}_{L^\infty(\Om^\eps_{\rm g})} \leq C/R$. Thus, the respective terms in \eqref{diffEq1} can be bounded as follows
\begin{align}\label{bound1}
\dfrac{C}{R}\left(\norm{\Na T^\eps}_{L^2(S\times \Om^\eps)}  + \norm{\Na \C^\eps}_{L^2(S\times \Om^\eps_{\rm g})}\right) \leq \dfrac{C}{R}.
\end{align}
The last inequality in \eqref{bound1} is a consequence of Lemma \ref{apriori}, which also imply that \eqref{convEq1} is bounded by $C/R.$ For sufficiently large $R$, \eqref{initEq1} tends to zero. Eventually, a simple change of coordinate and for sufficiently large $R$ we arrive at the desired result.  
\end{proof}
As next step, we establish an equicontinuity in time for the sequences of functions $(T^\eps, \C^\eps)$ in the moving coordinates. For this, we introduce the orthonormal basis $\{e_j\}_{j\in \mathbb{Z}^d} \in L^2(]0,1[^d)$, such that $e_j \in C^{\infty}_0([0,1]^d)$. Then, the functions $\{e_{jk}\}_{j,k\in \mathbb{Z}^d},$ where $e_{jk}(x) = e_j(x-k),$ form an orthonormal basis in $L^2(\Real^d).$ The equicontinuity in time for the sequences is stated in Lemma \ref{equiLemma}. The proof can be found in \ref{equiLemmaProof}.
%%%%%%%%%%%%%%%%%%%%%%%%%%%%%%%%%%%%%%%%%%%%%%%%%%%
\begin{lemma}\label{equiLemma}
Let $\delta t>0$ represent a small parameter for translation in time. Then, there exists a positive constant $C_{jk}$ independent of $\eps$ and $\delta t$ such that
\begin{align}
\label{equiLeEq0}
&\big|\big(\hat{T}^\eps(t + \delta t, \cdot), e_{jk}\big)_{L^2(\hat{\Om}^\eps(t + \delta t))} - \big(\hat{T}^\eps(t, \cdot), e_{jk}\big)_{L^2(\hat{\Om}^\eps(t))}\big| \leq C_{jk}\sqrt{\delta t}\\
\label{equiLeEq1}
&\big|\big(\hat{\C}^\eps(t + \delta t, \cdot), e_{jk}\big)_{L^2(\hat{\Om}^\eps_{\rm g}(t + \delta t))} - \big(\hat{\C}^\eps(t, \cdot), e_{jk}\big)_{L^2(\hat{\Om}^\eps_{\rm g}(t))}\big| \leq C_{jk}\sqrt{\delta t}.
\end{align}
\end{lemma}
To pass to the homogenization limit, we will make use of the following compactness theorem for the sequence of functions $(T^\eps, \C^\eps)$ in the moving coordinates.
\begin{theorem}\label{strongcomp}
There exists a subsequence $\eps$ and limit $(T^0,\C^0)\in L^2(S\times \Real^d) \times L^2(S\times \Real^d)$ such that 
\begin{align}
\label{STEq}
&\lim_{\eps\rightarrow 0}\int\limits_S\!\!\int\limits_{\hat{\Om^\eps(t)}}|\hat{T}^\eps(t,x)-T^0(t,x)|^2~dxdt=0,\\
\label{SCEq}
&\lim_{\eps\rightarrow 0}\int\limits_S\!\!\int\limits_{\hat{\Om^\eps_{\rm g}(t)}}|\hat{\C}^\eps(t,x)-\C^0(t,x)|^2~dxdt=0.
\end{align}
\end{theorem}
The proof of Theorem \ref{strongcomp} is given in \ref{Proofstrongcomp}.
%%%%%%%%%%%%%%%%%%%%%%%%%%%%%%%%%%%%%%%%%%%%%%%%%%%%%
% TWO SCALE CONVERGENCE WITH DRIFT
%%%%%%%%%%%%%%%%%%%%%%%%%%%%%%%%%%%%%%%%%%%%%%%%%%%%%
\section{Passage to the homogenization limit $\eps \to 0$}\label{rigor}
Using the two-scale convergence methods, we aim now to obtain the structure of the macroscopic combustion equations arising from our model $\mathcal{P}^\eps$ as $\eps \to 0.$ To this end, we employ the mathematical tool called the two scale convergence with drift, which is a suitable modification of the two-scale convergence concept (cf. e.g. \cite{Fatima12}) to account for suitable drifts. In this context, we follow the line of arguments from \cite{Harsha,Maro}. Also, we refer the reader to the PhD thesis \cite{HH} as well as to the references cited therein for more details and application examples of this averaging technique.

We recall the definition of the two-scale convergence with drift as stated in \cite[Proposition 3.1]{Harsha}.
\begin{definition}\label{2scale}
 Let $b^*$ be a constant vector in $\mathbb{R}^d$ and $u^\eps(t,x)\in L^2(S\times\mathbb{R}^d)$ be any bounded sequence of functions, i.e. there exists a constant $C>0$, independent of $\eps$, such that
 \begin{align*}
 \| u^\eps \|_{L^2(S\times \Real^d)} \leq C.
 \end{align*}
Then, there exists a function $u^0(t,x,y)\in  L^2(S\times\mathbb{R}^d\times \mathbb{T}^d)$ and up to the extraction of a subsequence \big(still denoted by $\eps$\big), the sequence $u^\eps$ two-scale converge with drift $b^*$ \big(or equivalently in the moving coordinates $(t,x)\to (t,x-b^*t/\eps)$\big) to $u^0$, in the sense that, 
 \begin{align}\label{conv_drift}
 \lim_{\eps\to 0}\int\limits_S\tg\int\limits_{\mathbb{R}^d}u^\eps(t,x)\varphi\Big(t, x-\frac{b^*}{\eps}t,\frac{x}{\eps}\Big)  dxdt=\int\limits_S\tg\int\limits_{\mathbb{R}^d}\tg\int\limits_{\mathbb{T}^d}u^0(t,x,y)\varphi(t,x,y)  dydxdt,
 \end{align}
 for all test functions $\varphi(t,x,y)\in \C_0^\infty(S\times\mathbb{R}^d\times \mathbb{T}^d)$. 
 \end{definition}
 We denote the convergence \eqref{conv_drift} by $u^\eps \twoscale  u^0$.  We refer to the constant vector $b^*$ as the effective drift. At this moment its choice is arbitrary. In $(ii)$, we will see that in our context the effective drift $b^*$ will be played by the vectors $b_T$ and $b_\C$ given by (\ref{drift1}) and (\ref{drift2}), respectively, depending whether we point out the temperature evolution or the evolution of the mass concentration.  
\begin{remark}
We point out that Definition \ref{2scale} is also applicable to a sequence of functions $u^\eps(t,x)\in L^2(S\times \Om^\eps)$ defined in a perforated domain $\Om^\eps$ satisfying the uniform bound $$\|u^\eps\|_{L^2(S\times \Om^\eps)} \leq C$$
such that
\begin{align*}
\lim_{\eps\to 0}\int\limits_S\tg\int\limits_{\Om^\eps}u^\eps(t,x)\psi(t,x-\dfrac{b^*}{\eps}t,\dfrac{x}{\eps})  dxdt = \int\limits_S\tg\int\limits_{\Real^d}\tg\int\limits_{Y}u^0(t,x,y)\psi(t,x,y)  dydxdt.
\end{align*}
\end{remark}
 Definition \ref{2scale} can be extended to sequences defined on periodic surfaces $\Gamma^\eps.$ In what follows, we give a statement of the two-scale convergence with drift on periodic surfaces due to \cite{Mik}.
 \begin{definition}\label{2scales}
 Let $b^*\in \Real^d$ and $w^\eps$ be a sequence in $L^2(S\times \Gamma^\eps)$ such that
 \begin{align*}
 \eps \int\limits_S\tg\int\limits_{\Gamma^\eps}|w^\eps(t,x)|^2 d\sigma(x)dt \leq C.
 \end{align*}
 Then, there exists a subsequence, still denoted by $\eps$, and a function $w^0(t,x,y)\in L^2(S\times \Real^d\times \Gamma)$ such that $w^\eps$ two-scale converge with drift $b^*$ to $w^0(t,x,y)$ in the sense that
 \begin{align*}
\lim_{\eps\to 0}&\eps  \int\limits_S\tg\int\limits_{\Gamma^\eps}w^\eps(t,x)\psi(t,x-\dfrac{b^*t}{\eps},\dfrac{x}{\eps}) d\sigma(x) dt\\
&= \int\limits_S\tg\int\limits_{\Real^d}\tg\int\limits_{\Gamma}w^\eps(t,x)\psi(t,x-\dfrac{b^*t}{\eps},\dfrac{x}{\eps}) dyd\sigma(y) dt
 \end{align*}
 for all $\psi(t,x,y)\in C^\infty_0(S\times \Real^d;C^\infty_\#(Y))$. 
 \end{definition}
 We denote the surface two-scale convergence with drift by $w^\eps \twoscales w^0.$
 %%%%%%%%%%%%%%%%%%%%%%%%%%%%%%%%%%%%%%%%%%%%%%%%%%%%%%%%%%%%%%%%%%%
\begin{theorem}\label{dr}
Let ${b}^*$ be a constant vector in $\mathbb{R}^d$ and $u^\eps$ a sequence of functions uniformly bounded in $L^2(S;H^1(\mathbb{R}^d))$. Then there exist a subsequence, still denoted by $\eps$, and functions $u^0(t,x)\in L^2(S;H^1(\mathbb{R}^d))$ and $u^1(t,x,y)\in L^2(S\times\mathbb{R}^d;H^1_\#(Y))$ such that
  \begin{equation}
  u^\eps \twoscale u^0
  \end{equation}
  and 
  \begin{equation}
  \Na u^\eps \twoscale \Na_x u^0+\Na_y u^1.
  \end{equation}
 \end{theorem}
 \begin{proof}
 We refer the reader to \cite{Maro} for the details of the proof and related context. 
 \end{proof}
 %%%%%%%%%%%%%%%%%%%%%%%%%%%%%%%%%%%%%%%%%%%%
 \subsection{Derivation of the macroscopic combustion equations}\label{main-result}
 %%%%%%%%%%%%%%%%%%%%%%%%%%%%%%%%%%%%%%%%%%%%
 The main result in this paper is the following strong convergence result:
 \begin{theorem}\label{MR} Assume $(H1)$--$(H6)$. The sequence $(T^\eps,\C^\eps)$ of solutions to problem $\mathcal{P}^\eps$ satisfies
 \begin{align}\label{bT}
 T^\eps(t,x) &= T^0(t,x-\frac{b_T}{\eps}t)+\rho_T^\eps(t,x),\\
 \C^\eps(t,x)&= \C^0(t,x-\frac{b_\C}{\eps}t)+\rho_\C^\eps(t,x),\label{bC}
 \end{align}
 with
 $$\lim_{\eps\to 0}\int\limits_S\int\limits_{\Om^\eps} |\rho_T^\eps(t,x)|^2dtdx=\lim_{\eps\to 0}\int\limits_S\int\limits_{\Om^\eps_{\rm g}} |\rho_\C^\eps(t,x)|^2dtdx=0.$$
In \eqref{bT} and \eqref{bC}, the vectors $b_T$ and $b_\C$ are the corresponding effective drifts given by 
\begin{align}
\label{drift1}
b_T := \dfrac{c_{\rm g}}{c^{\rm eff}}\int\limits_{Y_{\rm g}}b(y)dy
\end{align}
and
\begin{align}\label{drift2}
b_\C := \dfrac{1}{|Y_{\rm g}|}\int\limits_{Y_{\rm g}}b(y)dy,
\end{align}
while the pair $(T^0,\C^0)$ is the unique weak solution to the homogenized system
\begin{align}\label{homo}
\begin{dcases}
c^{\rm eff}\dfrac{\p T^0}{\p t}=\Na_x\cdot\!\big(\La^{\rm eff}\Na_xT^0\big) &\quad\mbox{in $S\times \Real^d$}\\
|Y_{\rm g}|\dfrac{\p \C^0}{\p t} = \Na_x\!\cdot\!\big(\D^{\rm eff}\Na_x\C^0\big) &\quad\mbox{in $S\times \Real^d$}\\
T_0(0,x)=T_{0}(x),~|Y_{\rm g}|\C_0(0,x)=|Y_{\rm g}|\C_{0}(x) &\quad\mbox{in $\Real^d$}.
\end{dcases}
\end{align}
Furthermore, the effective heat capacity is 
\begin{align}\label{effC}
c^{\rm eff}:=\int_{Y_{\rm g}}c_{\rm g} dy + \int_{Y_{\rm s}}c_{\rm s}dy,
\end{align}
while the entries of the dispersion tensors $\La^{\rm eff}$ and $\D^{\rm eff}$ are given respectively by
\begin{align}\label{Latensor}
\La^{\rm eff}_{ij}(T^0,\C^0)&:=\int\limits_{Y_{\rm g}}\La_{\rm g} \big(e_i + \Na_y\chi_{\textrm g,i}\big)\!\cdot\!\big(e_j + \Na_y\chi_{\textrm g,j}\big)dy\\\nonumber
&+\int\limits_{Y_{\rm s}}\La_{\rm s} \big(e_i + \Na_y\chi_{\textrm s,i}\big)\!\cdot\!\big(e_j + \Na_y\chi_{\textrm s,j}\big)dy\\\nonumber
&+QA\int\limits_{\Gamma}\Big[f'(T^0)\C^0\chi_i\chi_j + f(T^0)\omega_i\chi_j\Big]d\sigma(y),\\\nonumber
&\mbox{$i,j=1,\ldots,d,$}
\end{align}
and
\begin{align}\label{Dtensor}
\D^{\rm eff}_{ij}(T^0,\C^0)&:=\int\limits_{Y_{\rm g}}D(y)\big(e_i + \Na_y\omega_{i}\big)\!\cdot\!\big(e_j + \Na_y\omega_{j}\big)dy\\\nonumber
&-A\int\limits_{\Gamma}\Big[f(T^0)\omega_i\omega_j + \C^0f'(T^0)\omega_i\chi_j\Big]d\sigma(y),\\\nonumber
&\mbox{$i,j=1,\ldots,d$}.
\end{align}
Here \mbox{$(\chi,\omega) = \big(\chi_j,\omega_j\big)_{j=1,\ldots,d}$} is the solution of the cell problem 
%%%%%%%%%%%%%%%%%%%%%%%%%%%%%%%%%%%%%%%%
\begin{align}\label{cellproblem}
\begin{dcases}
c_{\rm g} b(y)\!\cdot\!\Na_y\chi_{\textrm{g},j} - \Na_y\!\cdot\!(\La_{\rm g} (\Na_y\chi_{\textrm g,j} + e_j) ) \\
= c_{\rm g} (b_T-b(y))\!\cdot\!e_j,&\mbox{in $Y_{\rm g},$}\\
-\Na_y\!\cdot\!(\La_{\rm s} (\Na_y\chi_{\textrm s,j} + e_j)) =c_{\rm s}b_T\!\cdot\!e_j,&\mbox{in $Y_{\rm s},$}\\
\chi_{\textrm g,j}-\chi_{\textrm s,i}=0,&\mbox{on $\Gamma,$}\\
\big[\La_{\rm s} (\Na_y\chi_{\textrm s,j}+e_j)-\La_{\rm g} (\Na_y\chi_{\textrm g,j} + e_j)\big]\!\cdot\! n\\
=QA\Big[f'(T^0)\C^0\chi_j + f(T^0)\omega_j\Big],&\mbox{on $\Gamma,$}\\
b(y)\!\cdot\!\Na_y\omega_{j} - \Na_y\!\cdot\!(D(y)(\Na_y\omega_{j} + e_j) ) = (b_\C-b(y))\!\cdot\!e_j,&\mbox{in $Y_{\rm g},$}\\
D(y)(\Na_y\omega_{j} + e_j)\!\cdot\! n=-A\Big[f'(T^0)\C^0\chi_j +f(T^0)\omega_j\Big],&\mbox{on $\Gamma,$}\\
y \rightarrow (\chi(y),\omega(y)) \mbox{  is $Y$-periodic.}
\end{dcases}
\end{align}
%%%%%%%%%%%%%%%%%%%%%%%%%%%%%%%%%%%%%%%%
with the first-order terms written as
\begin{align}\label{first1}
T^{1}(t,x,y) &= \chi(y,T^0,\C^0)\cdot\Na_x T^0(t,x),\\ 
\label{first2}
\C^{1}(t,x,y) &= \omega(y,T^0,\C^0)\cdot\Na_x \C^0(t,x).
\end{align}
\end{theorem} 
 We prove this result in the remainder of the paper, but before doing so we establish the two-scale compactness for the sequences introduced in Theorem \ref{MR}.
%%%%%%%%%%%%%%%%%%%%%%%%%%%%%%%%%%%%%%%%%%%%%
% TWO SCALE COMPACTNESS
%%%%%%%%%%%%%%%%%%%%%%%%%%%%%%%%%%%%%%%%%%%%%
\begin{theorem} 
Let $(T^\eps,\C^\eps)$  be a solution to $\mathcal{P}^\eps$ in the sense of Definition \ref{X1}. As $b^*$ for $T^\eps$ and $\C^\eps$, we assume the constant vectors ${b}_T$ and ${b}_\C$ defined in \eqref{drift1} and \eqref{drift2} respectively. Then, there exist subsequences $(T^\eps,\C^\eps)$, still denoted by $\eps$, and the pairs 
\begin{align}\label{energy}
&\mbox{$(T^0,T^1)$}  \in L^2(S;H^1(\mathbb{R}^d))\times L^2(S\times \mathbb{R}^d;H^1_\#(Y)),\\\nonumber
&\mbox{$(\C^0,\C^1)$}\in L^2(S;H^1(\mathbb{R}^d))\times L^2(S\times \mathbb{R}^d;H^1_\#(Y_{\rm g})),\nonumber
\end{align}
such that
\begin{align}\label{2scaleEqs}
\begin{dcases}
& T^\eps \twoscale  T^0(t,x),\\
& \C^\eps \twoscale \C^0(t,x),\\
& \Na T^\eps \twoscale \Na_x T^0(t,x)+\Na_y T^1(t,x,y),\\
& \Na \C^\eps\twoscale \Na_x \C^0(t,x)+\Na_y \C^1(t,x,y),\\
& \dfrac{1}{\eps}\C^\eps f(T^\eps) \twoscales \C^0f'(T^0)T^1(t,x,y) + f(T^0)\C^1(t,x,y).
\end{dcases}
\end{align}
\end{theorem}
%%%%%%%%%%%%%%%%%%%%%%%%%%%%%%%%%%%%%%%%%%%%%%%%%%%%%%
\begin{proof}
The statement is a straightforward Corollary of Theorem \ref{dr} given the $\eps$-independent {\em a priori} estimates stated in Lemma \ref{apriori} and Definitions \ref{2scale} and \ref{2scales}. Therefore, the sequences stated in \eqref{2scaleEqs} have two-scale limits with drift. The identification of the limits is rather standard. As guideline, see e.g. \cite{Fatima12}. However, we focus on showing how to derive the limit of the less obvious nonlinear limit in \eqref{2scaleEqs}, which is part of a sequence of steps towards establishing the homogenization results stated in Theorem \ref{MR}.
\section{Proof of Theorem \ref{MR}}\label{proofMR}
 We split the proof into four steps as given in the sequel.
%%%%%%%%%%%%%%%%%%%%%%%%%%%%%%%%%%%%%%%%%%%%%%%%%%%%%% STEPS 
%%%%%%%%%%%%%%%%%%%%%%%%%%%%%%%%%%%%%%%%%%%%%%%%%%%%%
\paragraph{$\rm I.$ Identification of the fast reaction limit}  In this step, we wish to handle the fast reaction term, which is potentially exploding when passing to the homogenization limit. In order to overcome this problem, inspired by a similar situation treated in \cite{Harsha}, we proceed as follows:
We take $\phi\in L^2(S\times \mathbb{R}^d\times \Gamma)$ such that 
$$\int\limits_{\Gamma}\phi (t,x,y)  d\sigma(y)=0,$$
for a.e. $(t,x)\in S\times \mathbb{R}^d$. Then, there exists a periodic vector field $\Theta\in L^2(S\times \mathbb{R}^d\times Y)^d$ such that
\begin{align}\label{fastaux}
\begin{dcases}
{\rm div}_y\Theta=0\quad&\mbox{in $Y,$}\\
\Theta\cdot n=\phi\quad&\mbox{on $\Gamma.$}
\end{dcases}
\end{align}
With \eqref{fastaux} at hand, we assume $\C^\eps$ is extended by zero to the whole of $\Om^\eps$ so that
$$
\eps\int\limits_S\tg\int\limits_{\Gamma^\eps} \dfrac{1}{\eps} W^\eps (\C^\eps,T^\eps)\phi\big(t,x-\dfrac{ b_T}{\eps}t,\dfrac{x}{\eps}\big)  d\sigma \\
=\int\limits_S\tg\int\limits_{\Omega^\eps} {\rm div}\Big(W^\eps (\C^\eps,T^\eps)\Theta\big(t,x-\dfrac{ b_T}{\eps}t,\dfrac{x}{\eps}\big)\Big)  dxdt
$$
$$=A\int\limits_S\tg\int\limits_{\Omega^\eps}\Big(\C^\eps f'(T^\eps)\Na T^\eps+\Na \C^\eps f(T^\eps)\Big)\cdot \Theta\big(t,x-\dfrac{ b_T}{\eps}t,\dfrac{x}{\eps}\big)  dxdt
$$
$$
+\int\limits_S\tg\int\limits_{\Omega^\eps} W^\eps (\C^\eps,T^\eps)~{\rm div}_x\Theta\big(t,x-\dfrac{ b_T}{\eps}t,\dfrac{x}{\eps}\big)  dxdt,
$$
which converges as $\eps\to 0$ to
$$
=A\int\limits_S\tg\int\limits_{\mathbb{R}^d}\tg\int\limits_Y\big(\C^0f'(T^0)(\Na_x T^0+\Na_y T^1)+(\Na_x \C^0+\Na_y \C^1)f(T^0))\cdot \Theta \big)  dydxdt
$$
$$
+A\int\limits_S\tg\int\limits_{\mathbb{R}^d}\tg\int\limits_{Y} \C^0f(T^0)~{\rm div}_x\Theta~dydxdt 
$$

$$
=A\int\limits_S\tg\int\limits_{\mathbb{R}^d}\tg\int\limits_{\Gamma} \big(\C^0f'(T^0)T^1 + f(T^0)\C^1\big)\Theta\cdot n~d\sigma(y)dxdt
$$

$$
=A\int\limits_S\tg\int\limits_{\mathbb{R}^d}\tg\int\limits_{\Gamma} \big(\C^0f'(T^0)T^1 + f(T^0)\C^1\big)\phi~d\sigma(y)dxdt
$$
%In order to obtain the last equality, we have used the periodicity in $y$ of the vector field $\Theta$. 
The convergence result relies on the essential fact that the sequence $\C^\eps f(T^\eps)$ has a two-scale limit, i.e. we used the strong compactness result established in Theorem \ref{strongcomp} and the fact that $f$ has at most a linear growth in $T$, i.e., \mbox{$f(T)\leq T$}, to deduce that $f(T^\eps)$ two-scale converges with drift to the limit $f(T^0)$, and hence to the two-scale with drift limit $\C^0 f(T^0)$ of the sequence $\C^\eps f(T^\eps)$.
\end{proof}
%%%%%%%%%%%%%%%%%%%%%%%%%%%%%%%%%%%%%%%%%%%%%%%%%%%%%%%%%%%%
\paragraph{$\rm II.$ Choice of the drifts} Now, we select the constant drift vector $b^*$ such that the cell problems are weakly solvable. 
\begin{lemma}\label{comp}
Assume $(H1)$--$(H6)$. Let $T^0(t,x)\geq 0$, $\C^0(t,x)\geq 0$ be given. Let the effective drifts $b_T$ and $b_\C$ respectively be defined as in \eqref{drift1} and \eqref{drift2} such that there exists a unique solution 
$$(\chi,\omega) = \big(\chi_j,\omega_j\big)_{j=1,\ldots,d}\in [H^1_\#(Y)]^d\times  [H^1_\#(Y_{\rm g})]^d,$$ 
up to the addition of a constant vector $(C, C)$ with $C\in \Real$, of the cell problem \eqref{cellproblem}.
\end{lemma}
%%%%%%%%%%%%%%%%%%%%%%%%%%%%%%%%%%%%%%%%
% PROOF OF THE CELL PROBLEMS
%%%%%%%%%%%%%%%%%%%%%%%%%%%%%%%%%%%%%%%%
\begin{proof}
One verifies directly that the corresponding compatibility condition is satisfied, i.e. by taking the average of the right hand side of \eqref{cellproblem} and equating the resulting expression to zero. This way, we deduce the structure of the effective drift velocities given in \eqref{drift1} and \eqref{drift2}. Furthermore, the combined variational formulation of \eqref{cellproblem} is 
\begin{align}\label{varfCell}
&\int\limits_{Y_{\rm g} }c_{\rm g} b(y)\cdot\Na_{y}\chi_{\textrm{g},j}\phi dy + \int\limits_{Y_{\rm g} }\La_{\rm g}\big(e_{j} + \Na_{y}\chi_{\textrm{g},j}\big)\cdot\Na_{y}\phi dy + \int\limits_{Y_{\rm s} }\La_{\rm s}\big(e_{j} + \Na_{y}\chi_{\textrm{s},j}\big)\cdot\Na_{y}\phi dy\\\nonumber
&+Q\int\limits_{Y_{\rm g} } b(y)\cdot\Na_{y}\omega_{j}\psi dy + Q\int\limits_{Y_{\rm g}}D\big(e_{j} + \Na_{y}\omega_{j}\big)\cdot\Na_{y}\psi dy \\\nonumber
&+ QA\int\limits_{\Gamma}\big[f'(T^0)\C^0\chi_{j} + f(T^0)\omega_{j}\big](\psi-\phi) d\sigma(y)\\\nonumber
&=\int\limits_{Y_{\rm g}}c_{\rm g}\big( b_{T}- b(y)\big)\cdot e_{j}\phi dy + \int\limits_{Y_{\rm s}}c_{\rm s} b_{T}\cdot e_j\phi dy+ \int\limits_{Y_{\rm g}}\big( b_{\C}- b(y)\big)\cdot e_{j}\psi dy,
\end{align}
which can be shown to satisfy the assumptions of the Lax-Milgram Lemma. 
\end{proof}
%%%%%%%%%%%%%%%%%%%%%%%%%%%%%%%%%%%%%%%%%%%%%%%%%%%%%%% PROOF OF THE HOMOGENIZED EQUATIONS
%%%%%%%%%%%%%%%%%%%%%%%%%%%%%%%%%%%%%%%%%%%%%%%%%%%%%
\paragraph{$\rm III.$ Derivation of the homogenized equations and effective coefficients} Take as test functions
$$(\hat\phi(t,x),\hat\psi(t,x))=\Big(\phi\Big(t,x-\dfrac{ b_T}{\eps}t\Big), \psi\Big(t,x-\dfrac{b_\C}{\eps}t\Big)\Big),$$
with $\phi,\psi\in C_0^\infty\big(\mathbb{R}^d\times S \big)$ with $\phi(t_f,x)=\psi(t_f,x)=0$. Also, observe that 
$$\dfrac{\p\phi}{\p t}(t,x)=\dfrac{\p\hat\phi}{\p t}(t,x) - \dfrac{ b_T}{\eps}\cdot \Na_x \hat\phi(t,x),$$
and correspondingly, 
$$\dfrac{\p\psi}{\p t}(t,x)=\dfrac{\p \hat\psi}{\p t}(t,x) - \dfrac{b_\C}{\eps}\cdot \Na_x \hat \psi(t,x).$$
Hence, using $(\hat \phi,Q\hat \psi)$ as a test function in the combined variational formulation of \eqref{eqmicro} yields:
\begin{align*}
&\int\limits_{S}\tg\int\limits_{\Om^\eps}c^\eps\dfrac{\p T^\eps}{\p t}\hat\phi~dxdt +
\dfrac{1}{\eps}\int\limits_S\tg\int\limits_{\Omega^\eps}c_{\rm g} b^\eps\cdot \Na T^\eps\hat\phi~dxdt
+\int\limits_{S}\tg\int\limits_{\Om^\eps}\La^\eps\Na T^\eps\cdot\Na\hat\phi~dxdt\\
&+Q \int\limits_S\tg\int\limits_{\Omega^\eps_{\rm g}}\dfrac{\p \C^\eps}{\p t}\hat\psi`dxdt
+\dfrac{Q}{\eps}\int\limits_S\tg\int\limits_{\Omega^\eps_{\rm g}}b^\eps\cdot \Na \C^\eps\hat\psi~dxdt+
Q\int\limits_S\tg\int\limits_{\Omega^\eps_{\rm g}}D^\eps\Na \C^\eps\cdot\Na\hat\psi~dxdt\\
&+\dfrac{Q}{\eps}\int\limits_S\tg\int\limits_{\Gamma^\eps} W^\eps (\hat{\psi}-\hat{\phi})~d\sigma^\eps(x)dt=0.
\end{align*}
Integrating by parts, with respect to time, in the last two identities, leads to
\begin{align*}
&-\int\limits_{\Om^\eps}c^\eps T^\eps(0,x)\hat\phi(0,x)~dx+\dfrac{1}{\eps}\int\limits_S\tg\int\limits_{\Omega^\eps_{\rm g}}c_{\rm g} (b_T- b^\eps)T^\eps\Na_x\hat\phi~dxdt -
\int\limits_S\tg\int\limits_{\Omega^\eps}c^\eps T^\eps\dfrac{\p\hat\phi}{\p t}~dxdt\\
&+\int\limits_S\tg\int\limits_{\Omega^\eps}\La^\eps\Na T^\eps\cdot\Na\hat\phi~dxdt
-Q\int\limits_{\Om^\eps_{\rm g}}\C^\eps(0,x)\hat\psi(0,x)~dx+\dfrac{Q}{\eps}\int\limits_S\tg\int\limits_{\Omega^\eps_{\rm g}} (b_\C- b^\eps)\C^\eps\Na_x\hat\psi~dxdt\\
&-Q\int\limits_S\tg\int\limits_{\Omega^\eps_{\rm g}}\C^\eps\dfrac{\p\hat\psi}{\p t}~dxdt
+Q\int\limits_S\tg\int\limits_{\Omega^\eps_{\rm g}}D^\eps\Na\C^\eps\cdot\Na\hat\psi~dxdt
+\dfrac{Q}{\eps}\int\limits_S\tg\int\limits_{\Gamma^\eps} W^{\eps}(\hat\psi-\hat\phi)~d\sigma^\eps(x)dt=0,
\end{align*}
It is worth pointing out that the terms 
$$
\dfrac{1}{\eps}\int\limits_S\tg\int\limits_{\Omega^\eps}c_{\rm g} (b_T-b^\eps)T^\eps\Na_x\hat\phi~dxdt~~\mbox{and}~~ \dfrac{Q}{\eps}\int\limits_S\tg\int\limits_{\Omega^\eps_{\rm g}}(b_\C-b^\eps)\C^\eps\Na_x\hat\psi~dxdt
$$ 
are potentially blowing up. Therefore, they need a special treatment. 
To handle them, we use two auxiliary classes of vector fields $\Pi$ and $\Sigma$, which we introduced earlier in \eqref{aux1} and \eqref{aux2}.
Thus, we obtain
\begin{align}
&-\int\limits_{\Om^\eps}c^\eps T^\eps(0,x)\hat\phi(0,x)~dx+
\eps\int\limits_S\tg\int\limits_{\Omega^\eps }\Delta \Pi_i^\eps \cdot\Na_x\hat\phi T^\eps  dxdt-
\int\limits_S\tg\int\limits_{\Omega^\eps}c^\eps T^\eps\dfrac{\p\hat\phi}{\p t}~dxdt\\\nonumber
&\int\limits_S\tg\int\limits_{\Omega^\eps}\La^\eps\Na T^\eps\cdot\Na\hat\phi~dxdt
-Q\int\limits_{\Om^\eps_{\rm g}}\C^\eps(0,x)\hat\psi(0,x)~dx +\eps Q\int\limits_S\tg\int\limits_{\Omega_{\rm g} ^\eps }\Delta \Sigma_i^\eps \cdot\Na_x\hat\psi \C^\eps  dxdt\\\nonumber
&-Q\int\limits_S\tg\int\limits_{\Omega^\eps_{\rm g}}\C^\eps\dfrac{\p\hat\psi}{\p t}~dxdt
+Q\int\limits_S\tg\int\limits_{\Omega^\eps_{\rm g}}D^\eps\Na\C^\eps\cdot\Na\hat\psi~dxdt
+\dfrac{Q}{\eps}\int\limits_S\tg\int\limits_{\Gamma^\eps} W^\eps (\hat\psi-\hat\phi)~d\sigma^\eps(x)dt=0,
\end{align}
which gives, after partial integration, the following expression: 
\begin{align}
\label{F1}
&-\int\limits_{\Om^\eps}c^\eps T^\eps(0,x)\hat\phi(0,x)~dx-\int\limits_S\tg\int\limits_{\Omega^\eps}c^\eps T^\eps\dfrac{\p\hat\phi}{\p t}~dxdt+\int\limits_S\tg\int\limits_{\Omega^\eps}\La^\eps\Na T^\eps\cdot\Na\hat\phi~dxdt\\\nonumber
&-\eps\int\limits_S\tg\int\limits_{\Omega^\eps } \sum_{i=1}^d\Na \Pi_i^\eps(x) \cdot\Na\Big(T^\eps  \partial_{x_i}\hat\phi \Big)~dxdt
+\eps Q\int\limits_S\tg\int\limits_{\Omega_{\rm g} ^\eps } \sum_{i=1}^d\Na\Sigma_i^\eps(x) \cdot\Na\Big(\C^\eps  \partial_{x_i}\hat\psi \Big)~dxdt\\\nonumber
&-Q\int\limits_{\Om^\eps_{\rm g}}\C^\eps(0,x)\hat\psi(0,x)~dx-Q\int\limits_S\tg\int\limits_{\Omega^\eps_{\rm g}}\C^\eps\dfrac{\p\hat\psi}{\p t}~dxdt+Q\int\limits_S\tg\int\limits_{\Omega^\eps_{\rm g}}D^\eps\Na\C^\eps\cdot\Na\hat\psi~dxdt\\\nonumber
&+\dfrac{Q}{\eps}\int\limits_S\tg\int\limits_{\Gamma^\eps} W^\eps (\hat\psi-\hat\phi)~d\sigma^\eps(x)dt=0,
\end{align}
Note that the terms  
$$\eps\int\limits_S\tg\int\limits_{\Omega^\eps } \sum_{i=1}^d\Na \Pi_i^\eps(x) \cdot\Na\Big(T^\eps  \p_{x_i}\hat\phi \Big)  dxdt
\quad\mbox{and}\quad 
 \eps\int\limits_S\tg\int\limits_{\Omega_{\rm g} ^\eps } \sum_{i=1}^d\Na \Sigma_i^\eps(x) \cdot\Na\Big(\C^\eps  \partial_{x_i}\hat\psi\Big)  dxdt
$$
converge two-scale with drift to the limits
\begin{align*}
\int\limits_S\tg\int\limits_{\mathbb{R}^d}\tg\int\limits_Y\sum_{i=1}^d\Na \Pi_i(y) \cdot\Na_yT^1(t,x,y)\partial_{x_i}\hat\phi(t,x,y)  dydxdt
\end{align*}
and
\begin{align*}
\int\limits_S\tg\int\limits_{ \mathbb{R}^d}\tg\int\limits_{Y_{\rm g}}\sum_{i=1}^d\Na\Sigma_i(y) \cdot\Na_y\C^1(t,x,y)  \partial_{x_i}\hat\psi(t,x,y)  dydxdt.
\end{align*}
Finally, to obtain the structure of the cell problems as well as the weak formulation of the limit equation, take as test function in  \eqref{F1} the expressions 
$$\hat\phi(t,x)=\phi_0\Big(t,x-\dfrac{ b_T}{\eps}t\Big) + \eps\phi_1\Big(t,x-\dfrac{ b_T}{\eps}t,\dfrac{x}{\eps}\Big),$$ and respectively, 
$$\hat\psi(t,x)=\psi_0\Big(t,x-\dfrac{ b_\C}{\eps}t\Big) + \eps\psi_1\Big(t,x-\dfrac{ b_\C}{\eps}t,\dfrac{x}{\eps}\Big),$$  
where $\big(\phi_0(t,x),\phi_1(t,x,y)\big)\in C_0^\infty(S\times \mathbb{R}^d)\times C_0^\infty(S\times \mathbb{R}^d;H_{\#}^1(Y))$\\
and $\big(\psi_0(t,x),\psi_1(t,x,y)\big)\in C_0^\infty(S\times \mathbb{R}^d)\times C_0^\infty(S\times \mathbb{R}^d;H_{\#}^1(Y_{\rm g}))$ 
satisfying 
\begin{align}\label{cond1}
\phi_i(t_f,x)=0=\psi_i(t_f,x).
\end{align}
In \cite{Harsha}, it was necessary to assume the same variable for the zeroth order terms of the oscillating test functions to deduce the desired coupling of the single physics problem at the macroscopic level. However, in the present multi-physics scenario, we assume the restriction of the zeroth order terms of the test functions on $\Gamma^\eps$ to be equal, i.e.,
\begin{align}\label{cond2}
 \phi_0(t,x)\Big|_{\Gamma^\eps}=\psi_0(t,x)\Big|_{\Gamma^\eps}~\mbox{for a.e. $t\in S.$}
\end{align}
It is now important to take as first step $\phi_0=\psi_0=0$ and then pass to the two-scale limit with drift in \eqref{F1}. We obtain the following weak forms
\begin{align}\label{Teq2}
&\int\limits_{S}\tg\int\limits_{\Real^d}\tg\int\limits_{Y}c(y)T^0 b_T\cdot\Na_x\phi_1  dydxdt + \int\limits_{S}\tg\int\limits_{\Real^d}\tg\int\limits_{Y_{\rm g}}c_{\rm g}  b(y)\cdot\big(\Na_xT^0 + \Na_yT^1\big)\phi_1~dydxdt \\\nonumber
&+\int\limits_{S}\tg\int\limits_{\Real^d}\tg\int\limits_{Y}\La(y)\big(\Na_xT^0 + \Na_yT^1\big)\cdot \Na_y\phi_1  dydxdt \\\nonumber
&- QA\int\limits_{S}\tg\int\limits_{\Real^d}\tg\int\limits_{\Gamma}
\big(\C^0f'(T^0)T^1+ f(T^0)\C^1\big)\phi_1~d\sigma(y)dxdt = 0.
\end{align}
Similarly, we obtain
\begin{align}\label{Ceq2}
&\int\limits_{S}\tg\int\limits_{\Real^d}\tg\int\limits_{Y_{\rm g} }\C^0b_\C\cdot\Na_x\psi_1  dydxdt + \int\limits_{S}\tg\int\limits_{\Real^d}\tg\int\limits_{Y_{\rm g} }
 b(y)\cdot\big(\Na_x\C^0 + \Na_y\C^1\big)\psi_1  dydxdt \\\nonumber
&+\int\limits_{S}\tg\int\limits_{\Real^d}\tg\int\limits_{Y_{\rm g} }D(y)\big(\Na_x\C^0 + \Na_y\C^1\big)\cdot \Na_y\psi_1  dydxdt \\\nonumber
&+ A\int\limits_{S}\tg\int\limits_{\Real^d}\tg\int\limits_{\Gamma}
\big(\C^0f'(T^0)T^1 + f(T^0)\C^1\big)\psi_1  d\sigma(y)dxdt = 0.
\end{align}
\eqref{Teq2} and \eqref{Ceq2} are simply the variational formulation of 
\begin{align}\label{corrector}
\begin{dcases}
-c_{\rm g} b_T\cdot \Na_xT^0 + c_{\rm g} b(y)\cdot\big(\Na_xT^0 + \Na_yT^1\big)\\
-{\rm div}_y\big(\La_{\rm g}(\Na_xT^0 + \Na_yT^1)\big) = 0,\quad&\mbox{in $Y_{\rm g}$}\\
-c_{\rm s} b_T\cdot\Na_xT^0-{\rm div}_y\big(\La_{\rm s}(\Na_xT^0 + \Na_yT^1)\big) = 0,\quad&\mbox{in $Y_{\rm s}$}\\
\big[T^1\big]_{\Gamma} = 0,\quad&\mbox{on $\Gamma$}\\
\big[\La_s\big(\Na_xT^0 + \Na_yT^1\big)-\La_{\rm g} \big(\Na_xT^0 + \Na_yT^1\big)\big]\cdot  n \\
= QA\big(\C^0f'(T^0)T^1 + f(T^0)\C^1\big),\quad&\mbox{on $\Gamma$}\\
-b_\C\cdot \Na_x\C^0 + b(y)\cdot\big(\Na_x\C^0 + \Na_y\C^1\big)\\
-{\rm div}_y\big(D(y)(\Na_x\C^0 + \Na_y\C^1)\big) = 0,\quad&\mbox{in $Y_g$}\\
D(y)\big(\Na_x\C^0 + \Na_y\C^1\big)\cdot n = -A\big(\C^0f'(T^0)T^1 + f(T^0)\C^1\big),\quad&\mbox{on $\Gamma$}\\
\mbox{$(T^1(y),\C^1(y))$ is $Y$-periodic.}  
\end{dcases}
\end{align}
\eqref{corrector} implies that
\begin{eqnarray*}
T^1(t,x,y) = \sum^{d}_{i=1}\chi_i(y)\dfrac{\p T^0}{\p x_i}(t,x),
\end{eqnarray*}
and 
\begin{eqnarray*}
\C^1(t,x,y) = \sum^{d}_{i=1}\omega_i(y)\dfrac{\p \C^0}{\p x_i}(t,x),
\end{eqnarray*}
where $\big(\chi_i,\omega_i\big),~{i=1,\ldots,d}$ solves the cell problem \eqref{cellproblem}.

\begin{remark}
The cell problems \eqref{cellproblem} \rm{(cf.~\eqref{corrector})} indicate that the convective transport is dominant at the microscopic (cell) level. This gives rise to the drifts exhibited by the macroscopic system in Eulerian coordinates. On the other hand, the assumption that the combustion is fast as described in $\mathcal{P}^\eps$ is also reflected on the structure of \eqref{cellproblem}.
\end{remark}

As second and last step, we take $\phi_1=\psi_1=0$ and pass again to the two-scale limit with drift in \eqref{F1}. This yields to the following limit equations:
\begin{align}\label{F2}
&\int\limits_S\tg\int\limits_{\Real^d}c^{\rm eff}\dfrac{\p T^0}{\p t}\phi_0 + \sum_{i,j=1}^{d}\int\limits_S\tg\int\limits_{\Real^d}\tg\int\limits_{Y}\La_{ij}(y)\dfrac{\p T^0}{\p x_j}\dfrac{\p \phi_0}{\p x_i} dydxdt \\\nonumber 
&+\int\limits_S\tg\int\limits_{\Real^d}\tg\int\limits_{Y}\ \sum_{i.j=1}^{d}\sum_{l=1}^{d}\La_{il}(y)\dfrac{\p\chi_{j}(y)}{\p y_l}\dfrac{\p T^0}{\p x_j}\dfrac{\p\phi_0}{\p x_i} dydxdt \\\nonumber
&+ \int\limits_S\tg\int\limits_{\Real^d}\tg\int\limits_{Y}\sum_{i.j=1}^{d}\sum_{l=1}^{d}\dfrac{\p\Pi_i(y)}{\p y_l}\dfrac{\p \chi_j(y)}{\p y_l}\dfrac{\p T^0}{\p x_j}\dfrac{\p\phi_0}{\p x_i}  dydxdt
\end{align}
\begin{align}\label{G2}
&+Q\Big\{\int\limits_S\tg\int\limits_{\Real^d}|Y_{\rm g}|\dfrac{\p \C^0}{\p t}\psi_0 + \sum_{i,j=1}^{d}\int\limits_S\tg\int\limits_{\Real^d}\tg\int\limits_{Y_{\rm g}}D_{ij}(y)\dfrac{\p \C^0}{\p x_j}\dfrac{\p \psi_0}{\p x_i} dydxdt \\\nonumber 
&+\int\limits_S\tg\int\limits_{\Real^d}\tg\int\limits_{Y_{\rm g}}\ \sum_{i.j=1}^{d}\sum_{l=1}^{d}D_{il}(y)\dfrac{\p\omega_{j}(y)}{\p y_l}\dfrac{\p \C^0}{\p x_j}\dfrac{\p\psi_0}{\p x_i} dydxdt \\\nonumber
&+ \int\limits_S\tg\int\limits_{\Real^d}\tg\int\limits_{Y_{\rm g}}\sum_{i.j=1}^{d}\sum_{l=1}^{d}\dfrac{\p\Sigma_i(y)}{\p y_l}\dfrac{\p \omega_j(y)}{\p y_l}\dfrac{\p \C^0}{\p x_j}\dfrac{\p\psi_0}{\p x_i}  dydxdt\Big\} = 0.
\end{align}
\eqref{F2} and \eqref{G2} are the variational formulation of the homogenized problem \eqref{homo} and the effective dispersion tensors are given by
%%%%%%%%%%%%%%%%%%%%%%
% Dispersion matrices
%%%%%%%%%%%%%%%%%%%%%%
\begin{align}\label{LeffEq0}
\La^{\rm eff}_{ij} = \int\limits_{Y}\La(y)e_i\cdot e_j  dy + \int\limits_{Y}\La(y)\Na_y\chi_j\cdot e_i  dy + \int\limits_{Y}\Na_y\Pi_i\cdot\Na_y\chi_j dy
\end{align}
and
\begin{align}\label{DeffEq0}
\D^{\rm eff}_{ij} = \int\limits_{Y_{\rm g}}D(y)e_i\cdot e_j  dy + \int\limits_{Y_{\rm g}}D(y)\Na_y\omega_j\cdot e_i  dy + \int\limits_{Y_{\rm g}}\Na_y\Sigma_i\cdot\Na_y\omega_j dy,
\end{align}
The coefficient $c^{\rm eff}$ arising in \eqref{F2} is defined by \eqref{effC}. In \eqref{F2} and \eqref{G2}, we identify the solutions to the auxiliary problems given in \eqref{aux1}, which simplify the singular terms in the variational formulations \eqref{Teq1} and \eqref{Ceq1}. We shall employ \eqref{aux1} presently in transforming the dispersion matrix \eqref{LeffEq0}, hence proving that it is equivalent to the formula \eqref{Latensor}. To achieve this, we test \eqref{aux1} for $\Pi_i$ by the cell solution $\chi_j$. Adding the resulting expressions lead to 
\begin{align}\label{LeffEq1}
\int\limits_{Y}\Na_y\Pi_i\cdot\Na_y\chi_j(y) dy = \int\limits_{Y}c(y)(b_{T,i}-b_i(y))~\chi_j(y) dy.
\end{align}
Substituting \eqref{LeffEq1} in \eqref{LeffEq0} gives the non-symmetrized form of the dispersion tensor, which can be obtained by means of the two-scale expansion with drift \cite{Allaire10}. In a next step, we replace the test functions in the variational formulation of \eqref{cellproblem} for $\big(\chi_i,\omega_i\big)$ by $\big(\chi_j,\omega_j\big)$. This results in the following 
\begin{align}\label{LeffEq2}
&\int\limits_{Y} c(y)(b_{T,i}-b_i(y))~\chi_j(y) dy = \int\limits_{Y}\lambda(y)\Na_y\chi_i\cdot\Na_y\chi_j dy\\\nonumber
&+ \int\limits_{Y}\lambda(y)\Na_y\chi_j\cdot e_i dy
+ QA\int\limits_{\Gamma} \Big(f'(T^0)\C^0\chi_i + f(T^0)\omega_i\Big)~\chi_j(y) d\sigma(y).
\end{align}
After substituting \eqref{LeffEq2} in \eqref{LeffEq0}, we see that the formulas \eqref{Latensor} and \eqref{LeffEq0} for the dispersion tensor $\La^{\rm eff}$ are equivalent. A similar argument leads to the dispersion tensor $\D^{\rm eff}$. Thus, we have obtained the variational formulation of the homogenized problem \eqref{homo}, which, according to Lemma \ref{unique}, admits a unique solution. It should be noted that as a consequence of the uniqueness of the solutions, the entire sequence converges.

\paragraph{$\rm IV.$ Uniqueness of solutions to the homogenized equations}
\begin{lemma}\label{unique}
Under the assumptions of Lemma \ref{comp}, there exists a unique solution of the couple, $(T^0,\C^0)$ such that
$$
(T^0,\C^0) \in C(\bar{S};L^2(\Real^d))\times C(\bar{S};L^2(\Real^d))
$$ 
and
$$
(\Na T^0,\Na\C^0) \in L^2(S\times\Real^d)\times L^2(S\times\Real^d)
$$ 
\end{lemma}
\begin{proof}
By construction, the symmetric part of the dispersion tensor satisfying \eqref{varfCell}
is given by 
\begin{align}\label{addtensor}
\mathfrak{L}_{ij}^{\rm sym}(T^0,\C^0) &= \int\limits_{Y_{\rm g}}\La_{\rm g} \big(e_i + \Na_y\chi_{\textrm{g},i}\big)\!\cdot\!\big(e_j + \Na_y\chi_{\textrm{g},j}\big)dy\\\nonumber
&+ \int\limits_{Y_{\rm s}}\La_{\rm s} \big(e_i + \Na_y\chi_{\textrm{s},i}\big)\!\cdot\!\big(e_j + \Na_y\chi_{\textrm{s},j}\big)dy\\\nonumber
&+Q\int\limits_{Y_{\rm g}}D(y)\big(e_i + \Na_y\omega_{i}\big)\!\cdot\!\big(e_j + \Na_y\omega_{j}\big)dy\\\nonumber
&+QA\Big(f(T^0)-f'(T^0)\C^0\Big)\int\limits_{\Gamma}\dfrac{\big(\omega_j\chi_i +\omega_i\chi_j\big)}{2}  d\sigma(y) \\\nonumber
&+f'(T^0)\C^0\int\limits_{\Gamma}\chi_{i}\chi_{j}  d\sigma(y)-f(T^0)\int\limits_{\Gamma}\omega_i\omega_j  d\sigma(y).
\end{align}
Since $f(T^0)\leq T^0$ and $Q>0$, we have 
\begin{align*}
\D^{\rm eff}(T^0,\C^0)\leq \int\limits_{Y_{\rm g}}D(y)dy,~~\La^{\rm eff}(T^0,\C^0) \leq C,
\end{align*} 
for some constant $C\in (0,\infty)$. Hence, $\D^{\rm eff}(T^0,\C^0)$ and $\La^{\rm eff}(T^0,\C^0)$ are uniformly bounded. Given that the diffusion tensors \eqref{Latensor} and \eqref{Dtensor} are symmetric and $f(T^0)\geq 0$, we also have that 
\begin{eqnarray}\label{Llbound}
\La^{\rm eff}(T^0,\C^0) \geq \Bigg[\int\limits_{Y_{\rm g}}\La_{\rm g} dy + \int\limits_{Y_{\rm s}}\La_{\rm s}  dy \Bigg]>\La_0,
\end{eqnarray}
for some $\La_0\in(0,\infty)$. \eqref{addtensor} also implies that 
\begin{eqnarray*}
\mathfrak{L}(T^0,\C^0) \geq  \int\limits_{Y_{\rm g}}\La_{\rm g} dy + \int\limits_{Y_{\rm s}}\La_{\rm s}  dy,
\end{eqnarray*} 
for $Q>0$ and $f(T^0)\geq 0$. Using the fact that $\mathfrak{L}=\La^{\rm eff} + Q\D^{\rm eff}$ and the estimate \eqref{Llbound} with $$\La_0= \max{\Bigg(\int\limits_{Y_{\rm g}}\La_{\rm g} dy,\int\limits_{Y_{\rm s}}\La_{\rm s}  dy\Bigg),}$$ then we observe $\D^{\rm eff}(T^0,\C^0)$ is bounded from below by
 \begin{eqnarray}\label{Dlbound}
\D^{\rm eff}(T^0,\C^0) \geq \dfrac{1}{Q}\Bigg[\int\limits_{Y_{\rm g}}\La_{\rm g} dy + \int\limits_{Y_{\rm s}}\La_{\rm s}  dy-\La_0\Bigg].
\end{eqnarray}
Thus, $\La^{\rm eff}(T^0,\C^0)$ and $\D^{\rm eff}(T^0,\C^0)$ are uniformly coercive and hence the desired uniqueness follows by standard arguments for parabolic equations.
\end{proof}

\section*{Acknowledgments}
A.M. thanks W. J\"ager (Heidelberg) for drawing his attention to the upscaling of filtration combustion and to G. Allaire (Paris) for a fruitful discussion on the concept of two scale convergence with drift and its potential applications to handling averaging scenarios involving reactive flow in porous media. We acknowledge financial supports from NWO MPE "Theoretical estimates of heat losses in geothermal wells" (grant nr. 657.014.004) and 
Science Foundation Ireland (SFI) under Grant Number 14/SP/2750.

\appendix
%%%%%%%%%%%%%%%%%%%%%%%%%%%%%%%%%%%%%%%%%%%%%%%%%%
%	APPENDIX
%%%%%%%%%%%%%%%%%%%%%%%%%%%%%%%%%%%%%%%%%%%%%%%%%%
% EQUICONTINUITY IN TIME OF THE SEQUENCE
%%%%%%%%%%%%%%%%%%%%%%%%%%%%%%%%%%%%%%%%%%%%%%%%%%
\section{Proof of the equicontinuity in time for the sequences of functions $(T^\eps, \C^\eps)$}\label{equiLemmaProof}
\begin{proof}
Let us denote $$\check{e}_{jk}(t,x)=e_{jk}\Big(x-\frac{b_{i}}{\eps}t\Big),~\mbox{$i=T,\C$}$$ and use $(\check{e}_{jk}, Q\check{e}'_{jk})$ as a test function in the variational formulation of $\mathcal{P}^\eps$. First, we start with the sequences of functions defined in the moving coordinates and compute the difference
\begin{align}\label{equiEq1}
& \Big(\hat{T}^\eps(t+\delta t, x), e_{jk}(x)\Big)_{L^2(\hat{\Om}^\eps(t + \delta t))} - \Big(\hat{T}^\eps(t, x), e_{jk}(x)\Big)_{L^2(\hat{\Om}^\eps(t))}\\\nonumber
&+ Q\Big(\hat{\C}^\eps(t+\delta t, x), e_{jk}(x)\Big)_{L^2(\hat{\Om}^\eps_{\rm g}(t + \delta t))} - Q\Big(\hat{\C}^\eps(t, x), e_{jk}(x)\Big)_{L^2(\hat{\Om}^\eps_{\rm g}(t))}\\\nonumber
&=\int\limits_t^{t +\delta t} \dfrac{d}{ds}\Big\{\int\limits_{\hat{\Om}^\eps(s)}\hat{T}^\eps(s,x)e_{jk}(x)~dx +Q \int\limits_{\hat{\Om}^\eps_{\rm g}(s)}\hat{\C}^\eps(s,x)e_{jk}(x)~dx\Big\}ds.
\end{align}
In the next step, we change to the fixed coordinates by setting $x \mapsto x -b_it/\eps$ on the right hand side of \eqref{equiEq1}. Then, we recall the variational formulation in fixed coordinates
\begin{align}\label{equiEq2}
&=\int\limits_t^{t +\delta t}\!\!\int\limits_{\Om^\eps}\Big\{\dfrac{\p T^\eps}{\p s}(s,x)\check{e}_{jk}(x) -\dfrac{b_T}{\eps}\cdot\Na\check{e}_{jk}(x)T^\eps(s,x)\Big\}dxds\\\nonumber
&+Q \int\limits_t^{t +\delta t}\!\!\int\limits_{\Om^\eps_{\rm g}}\Big\{\dfrac{\p \C^\eps}{\p s}(s,x)\check{e}'_{jk}(x) -\dfrac{b_\C}{\eps}\cdot\Na\check{e}'_{jk}(x)\C^\eps(s,x)\Big\}dxds\\\nonumber
&= \int\limits_t^{t +\delta t}\!\!\int\limits_{\Om^\eps}\Big(\dfrac{b^\eps-b_{T}}{\eps}\Big)\cdot\Na \check{e}_{jk}(x)T^\eps(s,x)~dxds -\int\limits_t^{t +\delta t}\!\!\int\limits_{\Om^\eps}\La^\eps\Na T^\eps\cdot\Na \check{e}_{jk}(x)~dxds\\\label{equiEq3}
&+ Q\int\limits_t^{t +\delta t}\!\!\int\limits_{\Om^\eps_{\rm g}}\Big(\dfrac{b^\eps-b_{\C}}{\eps}\Big)\cdot\Na \check{e}'_{jk}(x)\C^\eps(s,x)~dxds -Q\int\limits_t^{t +\delta t}\!\!\int\limits_{\Om^\eps_{\rm g}}D\Na \C^\eps\cdot\Na \check{e}'_{jk}(x)~dxds\\\nonumber
&+\dfrac{QA}{\eps} \int\limits_t^{t +\delta t}\!\!\int\limits_{\Gamma^\eps}\C^\eps f(T^\eps)(\check{e}_{jk}-\check{e}'_{jk})~d\sigma ds\nonumber
\end{align}
To deal with the singularity of the convective terms, we apply the auxiliary equations \eqref{aux2} on the convective terms on the right hand side of \eqref{equiEq2}
\begin{align}
\label{CjkEq1}
&= \int\limits_t^{t +\delta t}\!\!\int\limits_{\Om^\eps}\Big\{\eps\Delta \Pi^\eps_i(x)\p_{x_i}\check{e}_{jk}T^\eps(s,x)-\La^\eps\Na T^\eps(s,x)\cdot\Na\check{e}_{jk}(x)\Big\}~dxds\\
\label{CjkEq2}
&+\int\limits_t^{t +\delta t}\!\!\int\limits_{\Om^\eps_{\rm g}}\Big\{\eps\Delta \Sigma^\eps_i(x)\p_{x_i}\check{e}'_{jk}\C^\eps(s,x)-D^\eps\Na \C^\eps(s,x)\cdot\Na\check{e}'_{jk}(x)\Big\}~dxds\\
\label{CjkEq3}
&+\dfrac{QA}{\eps} \int\limits_t^{t +\delta t}\!\!\int\limits_{\Gamma^\eps}\C^\eps f(T^\eps)(\check{e}_{jk}-\check{e}'_{jk})~d\sigma ds.
\end{align}
%Note that the singularity of the convective terms has been dealt with by using the auxiliary equations \eqref{aux2}. 
%%%%%%%%%%%%%%%%%%%%%%%%%%%%%%%%%%%%%%%%%%%%%%%%%%
% ALTERNATIVE
%%%%%%%%%%%%%%%%%%%%%%%%%%%%%%%%%%%%%%%%%%%%%%%%%%
In order to handle the nonlinear integral \eqref{CjkEq3} with the coefficient $\eps^{-1}$, we state the following technical lemma
\begin{lemma}\label{auxBound}
Let $\phi(t,x,y)\in L^2(S\times\Real^d\times\Gamma)$ be such that 
\begin{eqnarray}\label{solvable}
\int\limits_{\Gamma}\phi(t,x,y) d\sigma = 0,
\end{eqnarray}
\mbox{for a.e. $(t,x)\in S\times \Real^d.$} There exist a periodic vector field $\Upsilon(t,x,y)\in L^2(S\times\Real^d\times \Gamma)^d$ such that ${\rm div}^s_y\Upsilon = \phi$~\mbox{on $\Gamma,$} where ${\rm div^s}_y = {\rm div}_yM(y)$ is the tangential divergence with $M(y)={\rm Id}-n(y)\otimes n(y)$, the projection matrix on the tangent hyperplane to the surface $\Gamma.$
\end{lemma}
\begin{proof} Let $\Upsilon=\Na^s_y\vartheta$ where $\vartheta$ is the unique solution in $H^1_{\#}(\Gamma)/\Real$ of $\Delta^s_y\vartheta = \phi$~\mbox{on $\Gamma$}, which satisfies the solvability condition due to \eqref{solvable}.
\end{proof}
To apply Lemma \ref{auxBound}, we assume restrictions of the basis functions $\check{e}_{jk}$ on $Y$ by a change of variable, $x=\eps y$, taking $\phi(t,y)=\check{e}_{jk}(t,y)-\check{e}'_{jk}(t,y)$. Then, for $\Om^\eps\subset\Real^d$, this translates to
\begin{align}\label{auxBoundEq2} 
\begin{dcases}
\eps~{\rm div}^s_x\Upsilon^\eps(x) = \phi(t,x)=\check{e}_{jk}(t,x)-\check{e}'_{jk}(t,x)~\mbox{on $\Gamma^\eps,$}\\
\mbox{$\Upsilon^\eps(x)$ is $\eps$-periodic.}
\end{dcases}
\end{align}
Integrating \eqref{CjkEq1} and \eqref{CjkEq2} by parts, applying \eqref{auxBoundEq2} on \eqref{CjkEq3} and introducing the nonlinear term $\W^\eps = \eps^{-1}\C^\eps f(T^\eps)$, we have
\begin{align}
\label{CjkEq4}
&= -\int\limits_t^{t +\delta t}\!\!\int\limits_{\Om^\eps}\Big\{\Na_y\Pi_i(y)\cdot\Na \big(\p_{x_i}\check{e}_{jk}T^\eps(s,x)\big)+\La^\eps\Na T^\eps(s,x)\cdot\Na\check{e}_{jk}(x)\Big\}~dxds\\
\label{CjkEq5}
&-\int\limits_t^{t +\delta t}\!\!\int\limits_{\Om^\eps_{\rm g}}\Big\{\Na_y \Sigma^\eps_i(y)\cdot\Na\big(\p_{x_i}\check{e}'_{jk}\C^\eps(s,x)\big)+D^\eps\Na \C^\eps(s,x)\cdot\Na\check{e}'_{jk}(x)\Big\}~dxds\\
\label{CjkEq6}
&+\eps QA \int\limits_t^{t +\delta t}\!\!\int\limits_{\Gamma^\eps}\W^\eps{\rm div}^s_x \Upsilon^\eps(x)~d\sigma ds\\
\label{CjkEq7}
&\leq C_{jk}\sqrt{\delta t}.
\end{align}
The right hand side bound \eqref{CjkEq7} follows by virtue of the a priori estimates of Lemma \ref{apriori}. Thus, we have
\begin{align*}
&\Big|\int\limits_t^{t +\delta t} \dfrac{d}{ds}\int\limits_{\hat{\Om}^\eps(s)}\hat{T}^\eps(s,x)e_{jk}(x)~dxds\Big| +Q \Big|\int\limits_t^{t +\delta t} \dfrac{d}{ds}\int\limits_{\hat{\Om}^\eps_{\rm g}(s)}\hat{\C}^\eps(s,x)e'_{jk}(x)~dxds\Big|\\
& \leq C_{jk}\sqrt{\delta t}.
\end{align*}
\end{proof}
%%%%%%%%%%%%%%%%%%%%%%%%%%%%%%%%%%%%%%%%%%%%%%%%%%%%%%%%%%%
\section{Strong compactness in the moving coordinates framework}
%%%%%%%%%%%%%%%%%%%%%%%%%%%%%%%%%%%%%%%%%%%%%%%%%%%%%%%%%%%
To characterize the sequence of functions $(T^\eps,\C^\eps)$ defined on the periodic domain $\Om^\eps$ in $\Real^d$, we introduce extension operators, as discussed in \cite{Harsha}. Let $E^\eps: H^1(\Om^\eps)\rightarrow H^1(\Real^d)$ be such that there exists a constant $C,$ independent of $\eps$, such that for all functions $\phi^\eps\in H^1(\Om^\eps)$, with $E^\eps \phi^\eps\Big|_{\Om^\eps}=\phi^\eps$
\begin{align}\label{extEq1}
\norm{E^\eps\phi^\eps}_{L^2(\Real^d)} \leq C\norm{\phi^\eps}_{L^2(\Om^\eps)},~~\norm{\Na E^\eps\phi^\eps}_{L^2(\Real^d)} \leq C\norm{\Na\phi^\eps}_{L^2(\Om^\eps)}.
\end{align}
In order to prove compactness in the moving coordinates framework, we will make use of the following sequences of functions, which are decomposed in terms of the orthonormal basis $\{e_{jk}\}\in \Real^d$
\begin{align}
\label{decTEq1}
&\hat{T}^\eps(t,x) = \sum_{j\in\mathbb{N}}\sum_{k\in\mathbb{Z}^d}\mu^{\eps}_{T,jk}(t)e_{jk}(x)~\mbox{with}~\mu^{\eps}_{T,jk}(t)=\int\limits_{\hat{\Om^\eps(t)}}\hat{T}^\eps(t,x)e_{jk}~dx\\
\label{decTEq2}
&\widehat{E^\eps T^\eps}(t,x) = \sum_{j\in\mathbb{N}}\sum_{k\in\mathbb{Z}^d}\nu^{\eps}_{T,jk}(t)e_{jk}(x)~\mbox{with}~\nu^{\eps}_{T,jk}(t)=\int\limits_{\Real^d}\widehat{E^\eps T^\eps}(t,x)e_{jk}~dx.
\end{align}
Similar decompositions also hold for $\hat{\C}^\eps(t,x)$ and $\widehat{E^\eps\C^\eps}(t,x)$ with the corresponding time dependent Fourier coefficients $\mu^{\eps}_{\C,jk}(t)$ and $\nu^{\eps}_{\C,jk}(t)$ defined in $\hat{\Om}^\eps_{\rm g}$ and $\Real^d$ respectively. We start off by showing compactness for the Fourier coefficients as stated in Lemma \ref{limitLemma}.
\begin{lemma}\label{limitLemma}
There exists subsequences, still denoted by $\eps$, such that 
\begin{align*}
\mu^\eps_{T,jk} \rightarrow \mu_{T,jk}~&\mbox{in $L^2(S),~~$ for all $j\in\mathbb{N}, k\in \mathbb{Z}^d$},\\
\mu^\eps_{\C,jk} \rightarrow \mu_{\C,jk}~&\mbox{in $L^2(S),~~$ for all $j\in\mathbb{N}, k\in \mathbb{Z}^d$},
\end{align*}
for some $\mu_{T,jk}\in L^2(S)$, respectively $\mu_{\C,jk}\in L^2(S)$, which are the Fourier coefficients to some functions
\begin{align*}
T^0(t,x) = \sum_{j\in \mathbb{N}}\sum_{k\in\mathbb{Z}^d}\mu_{T,jk}(t)e_{jk}(x),~\C^0(t,x) = \sum_{j\in \mathbb{N}}\sum_{k\in\mathbb{Z}^d}\mu_{\C,jk}(t)e_{jk}(x)
\end{align*}
belonging to $L^2(S\times \Real^d)$.
\end{lemma}
\begin{proof}
The idea of the proof due to \cite{Harsha} is as follows. A direct integration of \eqref{equiLeEq0} and \eqref{equiLeEq1} of Lemma \ref{equiLemma} in time $t\in (0,t_f-\delta t)$ gives the Riesz-Fr\'echet-Kolmogorov (RFK) criterion for strong compactness in $L^1(S)$; see, e.g. \cite{Haim11}. Thus, for any $j\in \mathbb{N}, k\in \mathbb{Z}^d$, there exist a subsequence $\eps_{jk}\rightarrow 0$ and limits $\mu_{T,jk}\in L^1(S)$ and $\mu_{\C,jk}\in L^1(S)$ such that
\begin{eqnarray*}
\mu^{\eps_{jk}}_{T,jk} \rightarrow \mu_{T,jk}\in L^1(S),~\mu^{\eps_{jk}}_{\C,jk} \rightarrow \mu_{\C,jk}\in L^1(S).
\end{eqnarray*}
By virtue of Lemma \ref{apriori}, $\mu^\eps_{T,jk}$ (resp. $\mu^\eps_{\C,jk}$) are bounded in $L^\infty(S)$, and the RFK property holds in $L^p(S), 1\leq p < \infty.$
Thus, it is straightforward to show that $(T^0,\C^0)\in L^2(S\times \Real^d)\times L^2(S\times \Real^d).$
\end{proof}
%%%%%%%%%%%%%%%%%%%%%%%%%%%%%%%%%%%%%%%%%%%%%%
\subsection{Properties of the Fourier coefficients}
%%%%%%%%%%%%%%%%%%%%%%%%%%%%%%%%%%%%%%%%%%%%%%
In this step, we estimate the difference between the Fourier coefficients $\mu^\eps_{jk}$ and $\nu^\eps_{jk}$.
\begin{lemma}\label{fourCoeff}
Let $\theta=|Y_{\rm g}|/|Y|$. There exists a constant $C_{jk}$ independent of $\eps$ such that 
\begin{align}\label{estEq1}
\big| \mu^{\eps}_{T,jk}(t)-\nu^\eps_{T,jk}(t)\big| \leq C_{jk}\eps,~\big| \mu^{\eps}_{\C,jk}(t)-\theta\nu^\eps_{\C,jk}(t)\big| \leq C_{jk}\eps.
\end{align}
\end{lemma}
%%%%%%%%%%%%%%%%%%%%%%%%%%%%%%%%%%%%%%%%%%%%%%
\begin{proof}
Using the definitions \eqref{decTEq1}-\eqref{decTEq2} of the Fourier coefficients, we obtain
\begin{align}
\label{decEq3}
&\mu^{\eps}_{T,jk}(t)-\nu^{\eps}_{T,jk}(t) = \int\limits_{\Om^\eps(t)}\hat{T}^\eps(t,x)e_{jk}(x)dx -\theta\int\limits_{\Real^d}\widehat{E^\eps T^\eps}(t,x)e_{jk}(x)dx\\
\label{decEq4}
&-(1-\theta)\int\limits_{\Real^d}\widehat{E^\eps T^\eps}(t,x)e_{jk}(x)dx\\
\label{decEq5}
&=\int\limits_{\Real^d}E^\eps T^\eps_{\rm s}(t,x)\check{e}_{jk}(x)\Big(\chi_{\rm s}(x/\eps)-(1-\theta)\Big)dx\\
\label{decEq6}
& + \int\limits_{\Real^d}E^\eps T^\eps_{\rm g}(t,x)\check{e}_{jk}(x)\Big(\chi_{\rm g}(x/\eps)-\theta\Big)dx,
\end{align}
where $\chi_i(x/\eps)$ are the characteristic functions of $\Om^\eps_i$ or their equivalents $\chi_i(y)$ defined on $Y_i$ with $i=\{g,s\}.$ To simplify \eqref{decEq5}-\eqref{decEq6} further, we introduce the following auxiliary problem:
\begin{align}\label{auxEq3}
\begin{dcases}
-{\rm div}_y\big(\Na_y\Psi(y)\big) = \chi_{\rm s}(y)-(1-\theta),~&\mbox{in $Y_{\rm s}$}\\
-{\rm div}_y\big(\Na_y\Psi(y)\big) = \chi_{\rm g}(y)-\theta,~&\mbox{in $Y_{\rm g}$}\\
\mbox{$\Psi(y)$ is $Y$-periodic}.
\end{dcases}
\end{align}
Using \eqref{auxEq3} in \eqref{decEq3} and integrating by parts the resulting expression results to
\begin{align}
&\big|\mu^{\eps}_{T,jk}(t)-\nu^{\eps}_{T,jk}(t)\big| \leq \eps\int\limits_{\Real^d}\Big|\Na_x\Psi^\eps(x)\cdot\Na\Big(E^\eps T^\eps_{\rm s}(t,x)\check{e}_{jk}(x)\Big)\Big|dx\\
\label{decEq7}
&+ \eps\int\limits_{\Real^d}\Big|\Na_x\Psi^\eps(x)\cdot\Na\Big(E^\eps T^\eps_{\rm g}(t,x)\check{e}_{jk}(x)\Big)\Big|dx \leq C_{jk}\eps.
\end{align}
The desired inequality in \eqref{decEq7} follows by using properties \eqref{extEq1} of $E^\eps$ and Lemma \ref{apriori}. Following a similar approach as described above leads to the second estimate of the Fourier coefficients given in \eqref{estEq1}.
\end{proof}
Now, we state a technical result that describes how the modal series of the bounded sequences in $L^2(S;H^1(\Real^d))$ introduced above can be truncated.
\begin{lemma}\label{truncEq1}
Let $u^\eps(t,x)$ be a bounded sequence in $L^2(S;H^1(\Real^d)).$ Then for any $\delta>0$, there exists a $N(\delta)$ such that for all $\eps$
\begin{align}
\Big\|u^\eps\chi_{Q_{R(\delta)}} - \sum_{|k|\leq R(\delta)}\sum_{|j|\leq N(\delta)} \U^\eps_{jk}(t)e_{jk}\Big\|_{L^2(S\times \Real^d)} \leq \delta,
\end{align}
where $Q_{R(\delta)}$ is defined in \eqref{cube} and 
\begin{align}
\U^\eps_{jk}(t)=\int\limits_{\Real^d}u^\eps(t,x)e_{jk}(x) dx
\end{align}
are the time dependent Fourier coefficients of $u^\eps$ defined as in \eqref{decTEq2}.
\end{lemma}
%%%%%%%%%%%%%%%%%%%%%%%%%%%%%%%%%%%%%%%%%%%%%%%%%%%%%%%%%%%%
\begin{proof}
The proof is a consequence of Lemma \ref{restEq1}, the embedding $H^1(Q_{R(\delta)})\hookrightarrow L^2(Q_{R(\delta)})$ and the a priori bounds of Lemma \ref{apriori}. For more details, we refer the interested reader to \cite[Lemma~4. p.~398]{Maro} and \cite[Lemma 2.9]{Harsha}.
\end{proof}
%%%%%%%%%%%%%%%%%%%%%%%%%%%%%%%%%%%%%%%%%%%%%%%%%%%%%%%%%%%%%%%%%%%%%%%

%%%%%%%%%%%%%%%%%%%%%%%%%%%%%%%%%%%%%%%%%%%%%%%%%%%%%%%%%%%%%%%%%%%%%%%
% PROOF OF STRONG COMPACTNESS OF THE SEQUENCE
%%%%%%%%%%%%%%%%%%%%%%%%%%%%%%%%%%%%%%%%%%%%%%%%%%%%%%%%%%%%%%%%%%%%%%%
\subsection{Proof of Lemma \ref{strongcomp}}
\begin{proof}\label{Proofstrongcomp}
According to Lemma \ref{restEq1}, for any $\delta>0$ and for sufficiently large $R(\delta)>0$, we have
\begin{align}\label{SEq1}
\norm{\hat{T}^\eps -\hat{T}^\eps\chi_{Q_{R(\delta)}}}_{L^2(S\times \hat{\Om}^\eps(t))} \leq \dfrac{\delta}{5},~\norm{\hat{\C}^\eps -\hat{\C}^\eps\chi_{Q_{R(\delta)}}}_{L^2(S\times \hat{\Om}^\eps_{\rm g}(t))} \leq \dfrac{\delta}{5}.
\end{align}
For the extended sequences $\widehat{E^\eps T^\eps}$ (resp. $\widehat{E^\eps \C^\eps}$) in $\Real^d$, Lemma \ref{truncEq1} implies that for any $\delta>0,$ there exists $N(\delta)$ such that for $\eps>0$ small enough,
\begin{align}\label{SEq2}
\Big\|\widehat{E^\eps T^\eps}\chi_{Q_{R(\delta)}} - \sum_{|k|\leq R(\delta)}\sum_{|j|\leq N(\delta)} \nu^\eps_{T,jk}(t)e_{jk}(x)\Big\|_{L^2(S\times \Real^d)} \leq \dfrac{\delta}{5}.
\end{align}
Without loss of generality, we will assume henceforth that similar results apply for the sequence $\hat{\C}^\eps.$ Since by definition, $$\widehat{E^\eps T^\eps}\Big|_{Q_{R(\delta)}}= \hat{T}^\eps \chi_{Q_{R(\delta)}}.$$ 
\eqref{SEq2} reduces to
\begin{align}\label{SEq3}
\Big\|\hat{T}^\eps\chi_{Q_{R(\delta)}} - \sum_{|k|\leq R(\delta)}\sum_{|j|\leq N(\delta)} \nu^\eps_{T,jk}(t)e_{jk}(x)\Big\|_{L^2(S\times \Real^d)} \leq \dfrac{\delta}{5}.
\end{align}
By Lemma \ref{equiLemma}, for any $\delta$ and $\eps$ small enough, we have
\begin{align}\label{SEq4}
\Big\| \sum_{|k|\leq R(\delta)}\sum_{|j|\leq N(\delta)}\nu^\eps_{T,jk}(t)e_{jk}(x)  - \sum_{|k|\leq R(\delta)}\sum_{|j|\leq N(\delta)} \mu^\eps_{T,jk}(t)e_{jk}(x)\Big\|_{L^2(S\times \Real^d)} \leq \dfrac{\delta}{5}.
\end{align}
Since by virtue of Lemma \ref{limitLemma}, $\mu^\eps_{T,jk}(t)$ is relatively compact in $L^2(S).$ Then, for sufficiently small $\eps$
\begin{align}\label{SEq5}
\Big\| \sum_{|k|\leq R(\delta)}\sum_{|j|\leq N(\delta)}\mu^\eps_{T,jk}(t)e_{jk}(x)  - \sum_{|k|\leq R(\delta)}\sum_{|j|\leq N(\delta)} \mu_{T,jk}(t)e_{jk}(x)\Big\|_{L^2(S\times \Real^d)} \leq \dfrac{\delta}{5}
\end{align}
Lemma \ref{limitLemma} implies the existence of a function $T^0(t,x) \in L^2(S\times \Real^d)$, such that for sufficiently large $N(\delta)$, we have
\begin{align}\label{SEq6}
\Big\| \sum_{|k|\leq R(\delta)}\sum_{|j|\leq N(\delta)}\mu_{T,jk}(t)e_{jk}(x)  - T^0(t,x)\Big\|_{L^2(S\times Q_{R(\delta)})} \leq \dfrac{\delta}{5}.
\end{align}
Finally, summing up \eqref{SEq1}-\eqref{SEq6}, we establish \eqref{STEq}, respectively \eqref{SCEq}.
\end{proof}


\begin{thebibliography}{10}
%\expandafter\ifx\csname url\endcsname\relax
%  \def\url#1{\texttt{#1}}\fi
%\expandafter\ifx\csname urlprefix\endcsname\relax\def\urlprefix{URL }\fi
%\expandafter\ifx\csname href\endcsname\relax
%  \def\href#1#2{#2} \def\path#1{#1}\fi
%
\bibitem{Moore90}
H.~Yi, J.~Moore, Self-propagating high-temperature (combustion) synthesis
  ({SHS}) of powder-compacted materials, Journal of Materials Science 25~(2)
  (1990) 1159--1168.

\bibitem{wang08}
S.~Wang, X.~Zhang, Microgravity smoldering combustion of flexible polyurethane
  foam with central ignition, Microgravity Sci. Technol. 20 (2008) 99--105.

\bibitem{Olson98}
S.~Olson, H.~Baum, T.~Kashiwagi, Finger-like smoldering over thin cellulose
  sheets in microgravity, Twenty-Seventh Symposium (International) on
  Combustion (1998) 2525--2533.

\bibitem{Aldushin06}
A.~P. Aldushin, A.~Bayliss, B.~J. Matkowsky, On the transition from smoldering
  to flaming, Combust. Flame 145 (2006) 579--606.

\bibitem{THULLIE95}
J.~Thullie, A.~Burghardt, Simplified procedure for estimating maximum cycling
  time of flow-reversal reactors, Chemical Engineering Science 50~(14) (1995)
  2299 -- 2309.

\bibitem{Andrez99}
A.~Burghardt, M.~Berezowski, E.~W. Jacobsen, Approximate characteristics of a
  moving temperature front in a fixed-bed catalytic reactor, Chem. Engineering
  and processing 38 (1999) 19 -- 34.

\bibitem{Faeth86}
G.~Faeth, G.~S. Samuelsen, Fast reaction nonpremixed combustion, Prog. Energy
  Combust. Sci. 12 (1986) 305--372.

\bibitem{Taylor53}
G.~Taylor, Dispersion of soluble matter in solvent flowing slowly through a
  tube, in: Proceedings of the Royal Society of London A: Mathematical,
  Physical and Engineering Sciences, Vol. 219, The Royal Society, 1953, pp.
  186--203.

\bibitem{Salles93}
J.~Salles, J.-F. Thovert, R.~Delannay, L.~Prevors, J.-L. Auriault, P.~Adler,
  Taylor dispersion in porous media. determination of the dispersion tensor,
  Physics of Fluids A: Fluid Dynamics (1989-1993) 5~(10) (1993) 2348--2376.

\bibitem{Fatehi94}
M.~Fatehi, M.~Kaviany, Adiabatic reverse combustion in a packed bed, Combust.
  Flame 99 (1994) 1--17.

\bibitem{Choquet14}
C.~Choquet, C.~Rosier, Effective models for reactive flow under a dominant
  {P}ecl{\'e}t number and order one {D}amk{\"o}hler number: Numerical
  simulations, Nonlinear Analysis: Real World Applications 15 (2014) 345 --
  360.

\bibitem{Pedras08}
M.~H. Pedras, M.~J. de~Lemos, Thermal dispersion in porous media as a function
  of the solid-fluid conductivity ratio, International Journal of Heat and Mass
  Transfer 51~(21-22) (2008) 5359 -- 5367.

\bibitem{HSU90}
C.~Hsu, P.~Cheng, Thermal dispersion in a porous medium, International Journal
  of Heat and Mass Transfer 33~(8) (1990) 1587 -- 1597.

\bibitem{Sano11}
Y.~Sano, F.~Kuwahara, M.~Mobedi, A.~Nakayama, Effects of thermal dispersion on
  heat transfer in cross-flow tubular heat exchangers, Heat and Mass Transfer
  48~(1) (2011) 183--189.

\bibitem{Moyne02}
C.~Moyne, S.~Didierjean, H.~A. Souto, O.~da~Silveira, Thermal dispersion in
  porous media: one-equation model, International Journal of Heat and Mass
  Transfer 43~(20) (2000) 3853 -- 3867.

\bibitem{Auriault95}
J.~L. Auriault, P.~Adler, {T}aylor dispersion in porous media: Analysis by
  multiple scale expansions, Advances in Water Resources 18~(4) (1995)
  217--226.

\bibitem{Harsha}
G.~Allaire, H.~Hutridurga, Upscaling nonlinear adsorption in periodic porous
  media--homogenization approach, Applicable Analysis 95~(10) (2016)
  2126--2161.

\bibitem{Kagan}
I.~Brailovsky, P.~V. Gordon, L.~Kagan, G.~Sivashinsky, Diffusive-thermal
  instabilities in premixed flames: Stepwise ignition-temperature kinetics,
  Combustion and Flame 162~(5) (2015) 2077--2086.

\bibitem{Fatima12}
T.~Fatima, A.~Muntean, Sulfate attack in sewer pipes : derivation of a concrete
  corrosion model via two-scale convergence, Nonlinear Anal.: Real World
  Applications 15 (2014) 326--344.

\bibitem{Muntean12}
T.~Fatima, A.~Muntean, M.~Ptashnyk, Unfolding-based corrector estimates for a
  reaction-diffusion system predicting concrete corrosion, Applicable Analysis
  91~(6) (2012) 1129--1154.

\bibitem{Maro}
E.~Marusik-Paloka, A.~Piatnitski, Homogenization of nonlinear
  convection-diffusion equation with rapidly oscillating coefficients and
  strong convection, J. London Math. Soc. 72~(2) (2005) 301--409.

\bibitem{Mik}
G.~Allaire, A.~Mikelic, A.~Piatnitski, Homogenization approach to the
  dispersion theory for reactive transport through porous media, SIAM J. Math.
  Anal. 42~(1) (2010) 125--144.

\bibitem{HH}
H.~Hutridurga, Homogenization of complex flows in porous media and
  applications, Ph.D. thesis, \'Ecole Polyt\'echnique, Palaiseau, France
  (2013).

\bibitem{Thomas17}
T.~Holding, H.~Hutridurga, J.~Rauch, Convergence along mean flows, SIAM Journal
  on Mathematical Analysis 49~(1) (2017) 222--271.

\bibitem{Allaire15}
G.~Allaire, H.~Hutridurga, On the homogenization of multicomponent transport,
  Discrete \& Continuous Dynamical Systems-B 20~(8) (2015) 2527--2551.

\bibitem{JFA}
G.~Allaire, I.~Pankratova, A.~Piatnitski, Homogenization and concentration of a
  diffusion equation with large convection in a bounded domain, Journal of
  Functional Analysis 262 (2012) 300--330.

\bibitem{HansBruining}
J.~Bruining, A.~A. Mailybaev, D.~Marchesin, Filtration combustion in wet porous
  medium, SIAM J. Appl. Math. 70~(4) (2009) 1157--1177.

\bibitem{buck85}
J.~D. Buckmaster, The {M}athematics of {C}ombustion, Frontiers in Combustion,
  Society of Industrial and Applied Mathematics, 1985.

\bibitem{Mimura09}
A.~Fasano, M.~Mimura, M.~Primicerio, Modelling a slow smoldering combustion
  process, Math. Methods Appl. Sci. (2009) 1--11.

\bibitem{Ijioma13}
E.~R. Ijioma, A.~Muntean, T.~Ogawa, Pattern formation in reverse smouldering
  combustion: {A} homogenisation approach, Combust. Theor. and Model. 17~(2)
  (2013) 185--223.

\bibitem{Ijioma15b}
E.~R. Ijioma, A.~Muntean, T.~Ogawa, Effect of material anisotropy on the
  fingering instability in reverse smoldering combustion, Int. J. Heat Mass
  Transfer 81~(0) (2015) 924--938.

\bibitem{Chen}
P.~J. Chen, M.~E. Gurtin, On a theory of heat conduction involving two
  temperatures, ZAMP 19~(4) (1968) 614--627.

\bibitem{Cioranescu99}
D.~Cioranescu, P.~Donato, An {I}ntroduction to {H}omogenization, Oxford
  University Press, New York, 1999.

\bibitem{Cioranescu07a}
D.~Cioranescu, P.~Donato, R.~Zaki, Asymptotic behavior of elliptic problems in
  perforated domains with nonlinear boundary conditions, Asymptotic Anal. 53
  (2007) 209--235.

\bibitem{Auriault91}
J.~L. Auriault, Heterogeneous medium. is an equivalent macroscopic description
  possible?, Int. J. Engng. Sci. 29~(7) (1991) 785--795.

\bibitem{Sivashinsky83}
G.~I. Sivashinsky, Instabilities, pattern formation, and turbulence in flames,
  Annual Review of Fluid Mechanics 15~(1) (1983) 179--199.

\bibitem{Hornung91}
U.~Hornung, W.~J\"ager, Diffusion, convection, adsorption and reaction of
  chemicals in porous media, J. Diff. Eqs. 92 (1991) 199--225.

\bibitem{Sara}
S.~Monsurro, Homogenization of a two-component composite with interfacial
  thermal barrier, Adv. Math. Sci. Appl. 13~(1) (2003) 43--63.

\bibitem{Allaire10}
G.~Allaire, R.~Brizzi, A.~Mikelic, A.~Piatnitski, Two-scale expansion with
  drift approach to the {T}aylor dispersion for reactive transport through
  porous media, Chem. Engrg. Sci. 65~(7) (2010) 2292--2300.

\bibitem{Haim11}
H.~Brezis, Functional {A}nalysis, {S}obolev {S}paces and {P}artial
  {D}ifferential {E}quations, Springer, Heidelberg, 2011.

\end{thebibliography}
\end{document}